\documentclass[12pt]{elsarticle}

\usepackage{graphicx}
\usepackage{amssymb}
\usepackage{amsthm}

\journal{arXiv}

\usepackage[ruled]{algorithm2e} 
\usepackage{amsmath}
\usepackage{subfig}

\usepackage{graphicx}

\newlength{\shor}
\def\tu#1{\langle #1\rangle}
\newcommand{\up}{\uparrow}
\newcommand{\down}{\downarrow}

\newcommand{\upts}{\uparrow}
\newcommand{\downts}{\downarrow}

\def\ess{\mathcal{E}}
\def\eess#1{\underline{#1}}

\newtheorem{theorem}{Theorem}
\newtheorem{lemma}{Lemma}
\newtheorem{observation}{Observation}
\newdefinition{remark}{Remark}
\newdefinition{example}{Example}

\newcommand{\Asso}{\textsc{Asso}}
\newcommand{\Hyper}{\textsc{Hyper}}

\newcommand{\GreConD}{\textsc{GreConD}}
\newcommand{\GreEss}{\textsc{GreEss}}
\newcommand{\Tiling}{\textsc{Tiling}}
\newcommand{\PaNDa}{\textsc{PaNDa}}
\newcommand{\ComputeIntervals}{\textsc{ComputeIntervals}}

\renewcommand*{\today}{June 10, 2013}

\begin{document}

\begin{frontmatter}



\title{From-Below Approximations in Boolean Matrix Factorization: Geometry and New Algorithm}


\author{Radim Belohlavek, Martin Trnecka}

\address{Data Analysis and Modeling Lab\\
Dept. Computer Science, Palack\'y University,  Czech Republic\\
%Palack\'y University, 17. listopadu 12, 771 46 Olomouc, Czech Republic\\
   e-mail: radim.belohlavek@acm.org, martin.trnecka@gmail.com
  }

\begin{abstract}
We present new results on Boolean matrix factorization and a new algorithm 
based on these results. 
The results emphasize the significance of factorizations that provide from-below 
approximations of the input matrix. While the previously proposed algorithms  
do not consider the possibly different significance of different matrix entries, 
our results help measure such significance and suggest where to focus when computing factors.
An experimental evaluation  of the new algorithm on both synthetic and real data demonstrates its good performance
in terms of good coverage by the first $k$ factors as well as a small number of factors
needed for exact decomposition and indicates that the algorithm outperforms the available ones
 in these terms.
We also propose future research topics.
\end{abstract}


\end{frontmatter}

% \linenumbers


\section{Introduction}

Boolean matrix factorization (BMF, called also Boolean matrix decomposition)
is becoming an established method for analysis and preprocessing of data. 
%RR: (CITACE?). 
The existing BMF methods are based on various types of heuristics and
approximation techniques, the fundamental reason being that the main
computational problems involved are known to be provably hard.
%RR: (NP-complete 
%decision problems, NP-hard or non-approximable within a constant factor optimization
%problems, see Section \ref{sec:p}). 
The heuristics employed, however,  use only a limited theoretical insight 
regarding BMF. 
 We show in this paper that a better understanding of the 
geometry of Boolean data results in a better understanding of BMF,
 theoretically justified heuristics, and better algorithms.

In particular,
we present new results in BMF derived from 
examining the closure and order-theoretic structures related to Boolean data, namely
the lattice of all fixpoints (so-called concept lattice) of the Galois connections 
associated to the input matrix.
Such viewpoint makes explicit the essence of BMF as a covering problem
and emphasizes one type of factorizations we call  from-below
factorizations. 
%Of the two possible types of error corresponding to the 
%``uncovered'' (0 instead of 1) and ``overcovered'' (1 instead of 0) entries 
%of the input matrix, the from-below factorizations commit only the first.
Such factorizations and some related notions were examined in some previous papers,
see Section \ref{sec:rw}.
%R4: predchozi vetu pozdeji odstranit
%Our viewpoint makes it possible to examine the role of 
%input matrix entries in BMF. 
While all the existing BMF methods consider the entries containing 1s in the input matrix essentially 
equally important, we propose to differentiate the role of such entries.
% examine %the statistical significance of appropriate 
%characteristics 
%of the entries that differentiate their role.
In particular, we examine the entries that are essential for BMF in that
their coverage by factors guarantees exact decomposition of the input matrix $I$
by these factors.
Crucial in our approach are intervals in the concept lattice associated to $I$.
We show that every such interval contains just the factors covering a certain rectangle (block full of 1s) in $I$
and that the intervals form reasonable subspaces for the search of factors.
We present a new BMF algorithm which is based on these results and computes from-below factorizations.
It turns out from experimental evaluation on both synthetic and real data that on average and on most real datasets,
the new algorithm outperforms the existing BMF algorithms.
Moreover, we clarify some connections between the existing approaches to BMF 
%R4: odstraneno: and some related problems
and argue that the closure and order-theoretic structures utilized 
%R4: reformulace
in this paper, which
make transparent the geometry of BMF,  represent
a useful framework for a reasonable theoretical analysis of the various BMF problems.
% It turns out from experimental evaluation
%on both synthetic and real data that
%the new algorithm outperforms both the two basic existing algorithms,
%namely Algorithm~2 from \citep{BeVy:Dof}  and 
%\Asso \citep{MiMiGiDaMa:TDBP} in terms of quality of decomposition as well as time cost.
%
%For example,
%while \Asso commits both types of error and delivers therefore general types
%of decompositions, our experiments indicate that it is significantly outperformed 
%in terms of quality of decompositions by the new algorithm
%proposed in this paper as well as by Algorithm 2 from \citep{BeVy:Dof}, 
%even though the latter two  delived only the from-below approximations of the input matrix.
%
The paper is concluded by discussing future research topics.
%R4: zbytek odstranen
%, namely regarding
%further heuristics, general types 
%of decompositions %M: preklep committing
%committing both types of errors, and extension of the present
%methods for general types of data (ordinal and semiring-valued). 
%RR: , and three-way).

%%%%%%%%%%%%%%%%%%%%%%%%%%%%%%%%%%%%%%%%%%%%%%%%%%%%%%%
\section{Preliminaries and Related Work}
\label{sec:p}

%-------------------------------------------------------------
\subsection{Notation and Basic Notions}\label{sec:nbn}
%\paragraph{Notation and Basic Notions}
%RR: [DOPLNIT REFERENCE]
Throughout this paper, we 
\[
\text{denote by $I$  an $n\times m$ Boolean matrix,}
\]
interpreted primarily as an object-attribute incidence (hence the symbol $I$) matrix,
i.e. the entry $I_{ij}$ corresponding to the row $i$ and the column $j$ is either $1$
or $0$, indicating that the object $i$ does or does not have the attribute $j$. 
%We use $I$ because very often, the matrix describes an \emph{i}ncidence
%relation between $n$ objects, which correspond to the rows, and $m$ attributes,
%which correspond to the columns of $I$.
The set of all $n\times m$ Boolean matrices is denoted by $\{0,1\}^{n\times m}$.
The $i$th row and $j$th column vectors of $I$ are denoted by $I_{i\_}$ and $I_{\_j}$,
respectively.
A general aim in BMF is to find for a given $I\in\{0,1\}^{n\times m}$
(and possibly other given parameters) matrices $A\in\{0,1\}^{n\times k}$ and
$B\in\{0,1\}^{k\times m}$ for which 
\begin{equation} 
  \label{eqn:IAM}
  I \mbox{ (approximately) equals } A\circ B,
  \text{where}   (A\circ B)_{ij} = \max_{l=1}^k \min(A_{il},B_{lj}),
\end{equation}
i.e. $\circ$ is the Boolean matrix product.
% given by
%\begin{equation}
%  \label{eqn:BP} 
%  (A\circ B)_{ij} = \max_{l=1}^k \min(A_{il},B_{lj}).
%\end{equation}
A decomposition of $I$ into $A\circ B$ may be interpreted as a discovery of
$k$ factors that exactly or approximately explain the data: 
interpreting $I$, $A$, and $B$ as the object-attribute, object-factor, and factor-attribute
matrices, 
the model (\ref{eqn:IAM}) reads: the object $i$ has the attribute $j$ if and only if 
there exists factor $l$ such that $l$ applies to $i$ and $j$ is one of the particular manifestations
of $l$.
%$I$, $A$, and $B$ are naturally called the object-attribute matrix, the object-factor (or usage) matrix,
%and the factor-attribute (or basis vector) matrix \cite{BeVy:Dof,MiMiGiDaMa:TDBP}. 
%
%$A_{il}=1$ indicates that 
%factor $l$ applies to object $i$ while $B_{lj}$ indicates that attribute $j$ is a particular
%manifestation of factor $l$ (think of  person A as object, ``being fluent in English'' as attribute, 
%and ``having good education'' as factor). 
The least $k$ for which an exact decomposition $I=A\circ B$ exists is called the 
\textit{Boolean rank} (Schein rank) of $k$ and is denoted by $\mathrm{rank}_\mathrm{B}(I)$.
% \cite{BeVy:Dof,Kim:BMTA,MiMiGiDaMa:TDBP}.
%Then, according to (\ref{eqn:BP}),
%the factor model reads: object $i$ has attribute $j$ if and only if 
%there exists factor $l$ such that $l$ applies to $i$ and $j$ is a partcular manifestation
%of $l$.
%$I$, $A$, and $B$ are naturally called the object-attribute matrix, the object-factor (or usage) matrix,
%and the factor-attribute (or basis vector) matrix \cite{BeVy:Dof,MiMiGiDaMa:TDBP}. 

Recall that the $L_1$-norm (Hamming weight in case of Boolean matrices) 
$||\cdot||$ and the corresponding metric $E(\cdot,\cdot)$ are defined
for 
%RR: every 
$C,D\in\{0,1\}^{n\times m}$ by
% is given by $$|| C || =\textstyle\sum_{i=1,j=1}^{m,n} |C_{ij}|.$$
%The induced matrix distance, that is conveniently used for measuring
%error in matrix decompositions, is given by 
\begin{equation}\label{eqn:error}
%|| C || =\textstyle\sum_{i=1,j=1}^{m,n} |C_{ij}|
%\text{ and }
%  E(C,D)=\textstyle
%  ||C-D ||=\sum_{i=1,j=1}^{m,n} |C_{ij}-D_{ij}|.
|| C || =\textstyle\sum_{i,j=1}^{m,n} |C_{ij}|
\quad\text{ and }\quad
  E(C,D)=\textstyle
  ||C-D ||=\sum_{i,j=1}^{m,n} |C_{ij}-D_{ij}|.
\end{equation}
%
%NAKONEC UPRAVIT
The following variants of the BMF problem, relevant to this paper, are considered in the literature.
\begin{itemize}
  \item[--]
  \emph{Discrete Basis Problem} (DBP, \citep{MiMiGiDaMa:TDBP}):\\
   Given $I\in\{0,1\}^{n\times m}$ and a positive integer $k$,
   find $A\in\{0,1\}^{n\times k}$ and $B\in\{0,1\}^{k\times m}$ 
   that minimize $||I-A\circ B||$.
  \item[--]
%R4: pridana carka za AFP
  \emph{Approximate Factorization Problem} (AFP, \citep{BeVy:Dof}):\\ 
   Given $I$ and prescribed error $\varepsilon\geq 0$, 
  find $A\in\{0,1\}^{n\times k}$ and $B\in\{0,1\}^{k\times m}$ with $k$ as small as 
  possible such that $||I-A\circ B||\leq \varepsilon$. 
\end{itemize}
These two problems reflect two important views on BMF. The first one emphasizes
the importance of the first $k$ (presumably most important) factors. The second 
one emphasizes the need to account for (and thus to explain) a prescribed 
portion of data, which is specified by $\varepsilon$.

%-------------------------------------------------------------
\subsection{Related Work}\label{sec:rw}
%\paragraph{Related Work}
%RR:
%As matrix decomposition methods in general represent a rather broad and thoroughly studied
%subject whose overview is beyond the scope of this paper
%we refer the reader to [REFERENCE? \citep{Gol:MC}] and limit ourselves to the work in BMF directly related to 
%our paper.
Matrix decompositions represent an extensive subject whose coverage is beyond
the scope of this paper. A good overview from BMF viewpoint is found e.g. in
\citep{MiMiGiDaMa:TDBP}.
Except for the area of Boolean matrix theory itself, see e.g. \citep{Kim:BMTA}, 
relevant results are traditionally presented in the literature on
binary relations and graph theory, see e.g. \citep{BrRy:CMT,Sch:RM}. These results may be 
%M: preklep translated 
translated to the results on Boolean matrices
due to the various one-one correspondences between the involved notions, such as those connecting Boolean matrices,
bipartite graphs, and binary relations, and pertain mostly to combinatorial and
computational complexity questions.
An important related area is formal concept analysis (FCA) \citep{GaWi:FCA}, in which Boolean matrices are
represented by so-called formal contexts, i.e. binary relations between objects and attributes.
FCA provides solid lattice-theoretical foundations which are utilized in our paper.

Decompositions of Boolean matrices using decomposition methods designed originally
for real-valued data and various modifications of these methods appear in a number
of papers.
\citep{TaMiGiMa:Widybd}  compares several approaches to assessment of dimensionality
of Boolean data, concluding among other observations that a principal problem with
applying to Boolean data the methods designed originally for real-valued data is
the lack of interpretability. Similar observations were presented  by 
other authors as well, emphasizing the need for methods particularly %MM: tohle bylo odkomentovane, takze v textu bylo ceske slovo, asi omilem smazno tailo: preklep tailored
tailored to Boolean data.
Among the first works on applications of 
BMF involving the Boolean matrix product in data analysis are \citep{Nau:Sc,Naea:SMahlas}, 
in which the authors have already been aware of the
provable computational difficulty (NP-hardness) of the decomposition problem due to 
NP-hardness of the set basis problem \citep{Sto:Sbpinpc}.
%RR: podrobneji, Nau:Sc vi, ze zjistit, zda rankB<=k je NP-complete
%Jejich modifikace pro binarni data
%First works: \citep{Nau:Sc}, BMDP?
%Boolean matrices \cite{Kim:BMTA}

The interest in BMF in data mining is primarily due to the work of Miettinen et al.
In particular, the DBP, the corresponding complexity results,
and the %M: mezera za makrem, opraveno
\Asso\ algorithm discussed below appeared in \citep{MiMiGiDaMa:TDBP}. 
%The \Asso algorithm and various problems regarding BMF are discussed in several other
%papers co-authored by Miettinen, e.g. \citep{Mie:Sbmf,MiVr:Mosbmf}
In \citep{GeGoMi:Td}, they authors examine ``tiling'' of Boolean data
and various related problems, their complexity, and algorithms. 
Tiling is closely related to BMF as it corresponds to the from-below factorizations we investigate in this paper
and is discussed  in more detail in Section \ref{sec:ac}. 
In \citep{BeVy:Dof}, our previous paper, we showed how to use formal concepts (i.e. fixpoints of Galois connections) of 
Boolean matrices as factors,
proved their optimality for exact factorizations, described transformations between
attribute and factor spaces, 
%RR: made explicit a link to set-cover-like probems, 
and proposed 
two BMF algorithms discussed below.
In \cite{Xiea:Stdoh}, the authors investigate the problem of summarizing transactional databases
by so-called hyperrectangles, examine the computational complexity of the problems involved, provide 
%R4: misto: two algorithms called \Hyper\ and \HyperPlus, and discuss related problems. dano:
the \Hyper\ algorithm and discuss related problems.
The summarizations involved may be rephrased as Boolean matrix decompositions and this approach %R4: and the algorithms are -> is
is discussed in more detail in Sections \ref{sec:fba} and  \ref{sec:e}.
Directly relevant to our paper is also \cite{LuOrPe:Mtpbdn}, where the authors propose an algorithm, called \PaNDa,
for computing top-$k$ patterns in Boolean datasets. The algorithm employs the minimum %M: preklep description
description length principle
and is discussed in Section \ref{sec:e}.
In particular, we use the algorithms proposed in the above five papers, namely 
\citep{BeVy:Dof,GeGoMi:Td,LuOrPe:Mtpbdn,MiMiGiDaMa:TDBP,Xiea:Stdoh},
in the experimental evaluation of the algorithm proposed in our paper.

Further work relevant to BMF includes other Miettinen's papers, such as
\citep{Mi:Bccrmd} in which the Boolean CX and CUR decompositions,  their complexity, and algorithms
are studied,
\citep{Mi:Sbmf} which investigates the issue of sparsity in BMF,
 \citep{MiVr:Mosbmf} where authors propose a general strategy to employ the minimum description
length principle in BMF in selecting the number of factors and apply it to \Asso,
and \citep{Mie:Fjsbmf} which examines the problem of finding common 
%R4: factor -> factors
factors of two and more matrices.
Measuring differences between summarizations of data  with itemsets and tiles is an important topic, for which
we refer to \citep{Ta:Caomdbdmr} and the references therein.
\citep{Mpr:scbcfm} presents a useful survey containing several results on complexity
and various ranks for Boolean matrices which we do not address in detail in this paper. 
Regarding ranks, the reader is also referred to \citep{MiMiGiDaMa:TDBP}; for complexity
of the various problems related to BMF, the reader is referred to the above papers and to
\citep{VaAtGu:Trmpfmdsf}. 
%The provable computational difficulty of the various problems is 
%mostly proved by reduction to a well-known difficult problem such as the set basis problem 
%mentioned above or other ones \citep{Hro:AHP}.
%
In addition to the above works, interesting applications of BMF have recently been
presented to role mining \citep{Luea:Carmebmd} utilizing certain extensions of BMF,
see also \citep{Luea:Obmd,VaAtGu:Trmpfmdsf}, and reducing dimensionality in
 classification of Boolean data \citep{Out:Bfadpml}, resulting in improved classification
accuracy.
%
%Due to restricted scope, we omit other references to the existing, growing literature
%related to BMF.

%\citep{GiMaSe:Gctd}   BeVy:Dof,GeGoMi:Td,Mi:Sbmf,MiMiGiDaMa:TDBP
% \citep{Frea:Bfaann}
%? A Simple Algorithm for Topic Identification in 0-1 data \citep{SeBiMa:Satid}
%PROJIT: \citep{Lub:Bbpcspr}

%%%%%%%%%%%%%%%%%%%%%%%%%%%%%%%%%%%%%%%%%%%%%%%%%%%%%%%
\section{From-Below Approximations and Geometry of BMF}
\label{sec:fba}
%In this section, 
%we define and examine the from-below factorizations, present a geometric insight regarding 
%BMF and the from-below factorizations in terms of the associated closure and order-theoretic structures.

%-------------------------------------------------------------
%\subsection{Geometry of decomposition, two components of error,
%and from-below decompositions}
%\label{sec:gd}

\subsection{Factorizations as Coverings and From-Below Approximations}
%\label{sec:gd}
%\paragraph{Factorizations as Superpositions of Rectangles}
We first make explicit the following view of decompositions,
present implicitly in \citep{BeVy:Dof}. For matrices $J_1$ and $J_2$, we put
\begin{equation}\label{eqn:cont}
J_1\leq J_2 \mbox{ ($J_1$ is contained in $J_2$) \quad  if{}f } \quad (J_1)_{ij}\leq (J_2)_{ij}  \mbox{ for every $i,j$.}
\end{equation}
A matrix $J\in\{0,1\}^{n\times m}$ is called \emph{rectangular}
(a rectangle, for short) if 
$J=C\circ D$ for some  $C\in\{0,1\}^{n\times 1}$ (column)
and $D\in\{0,1\}^{1\times m}$ (row), i.e. $J$ is the cross-product 
of two vectors.
Clearly, this means that
upon suitable permutations of columns and rows the 1s in $J$
form a rectangular area.
% and that the rows and columns of this area 
%are just the sets $\{i \mid C_{i}=1\}$ and $\{j \mid D_j=1\}$.
We say that $J$ (or, the pair $\tu{C,D}$ for which $J=C\circ D$)
\emph{covers} $\tu{i,j}$ if $J_{ij}=1$ (equivalently, $C_i=1$ and $D_j=1$).

%Recall that the componentwise $\vee$-composition of Boolean matrices 
%is defined by $(C\vee D)_{ij}=\max(C_{ij},D_{ij})$ and that
%we put 
%$$C\leq D \mbox{ ($C$ is contained in $D$)  if{}f } C_{ij}\leq D_{ij}  \mbox{ for every $i,j$.}
%$$
%The folowing easy-to-see observation provides an alternative useful 
%way to look at Boolean matrix (de)compositions.
%
%The following lemma presents a geometric interpretation of a Boolean matrix product,
%namely, a product $A\circ B$ may be lookd at as a $\vee$-composition
%of rectangles contained in $I$
%(note that a part of this observation is present in the proof of \cite[Theorem 2]{BeVy:Dof}).
%
%\smallskip

\begin{observation}
   \label{thm:rec}
  The following conditions are equivalent for any $I\in\{0,1\}^{n\times m}$.
  \begin{itemize}
     \item[\emph{(a)}]
     $I=A\circ B$ for some $A\in\{0,1\}^{n\times k}$ and $B\in\{0,1\}^{k\times m}$.

     \item[\emph{(b)}]
     There exist rectangles $J_1,\dots,J_k\in \{0,1\}^{n\times m}$ such that
     $I=J_1\vee \cdots\vee J_k$, i.e.
     $I_{ij}=\max_{l=1}^k ({J_l})_{ij}$.

     \item[\emph{(c)}]
     There exist rectangles $J_1,\dots,J_k\in \{0,1\}^{n\times m}$ contained in $I$ such that
      $I_{ij}=1$ if and only if $\tu{i,j}$ is covered by some $J_l$.
\end{itemize}
%
%  if and only if $I$  is a $\max$-superposition
%  of rectangles $J_1,\dots,J_k$ contained in $I$, i.e. for every $i,j$ we have
%  $I_{ij}=\max_{l=1}^k ({J_l})_{ij}$. In particular, one may put
%   $J_l=A_{\_l}\circ B_{l\_}$.
\end{observation}
%\begin{proof}
%The proof is easy and is based on the fact that given $A$ and $B$,
%one may construct $J_l$ as $A_{\_l}\circ B_{l\_}$, and that
%given the $J_l$s, which are by definition of the form $J_l=C_l\circ D_l$, 
%one obtains $A$ and $B$ by taking $C_l$ and $D_l$ for the $l$th column
%and $l$th row of $A$ and $B$, respectively.
%\end{proof}
In particular, if $A$ and $B$ are the matrices from Observation \ref{thm:rec} (a) 
then one may put $J_l=A_{\_l}\circ B_{l\_}$ ($l=1,\dots,k$), i.e. $J_l$ is the
product of the $l$th column of $A$ and the $l$th row of $B$, to obtain the rectangles
in (b) and (c).
Conversely, if $J_1=C_1\circ D_1,\dots,J_k=C_k\circ D_k$, for some column and row vectors 
$C_l\in\{0,1\}^{n\times 1}$ and $D_l\in\{0,1\}^{1\times m}$, are the rectangles in (b) or (c) then 
the matrices $A$ and $B$ in which the $l$th column and $l$th row are $C_l$ and $D_l$, respectively,
satisfy (a).
Hence,
%Observation \ref{thm:rec} implies that
 if $A$ and $B$ form the output of any BMF method
%R4: corresponding to the -> for an
for an input
matrix $I$, one
%R4 reformulace
may identify the factors $l=1,\dots,k$
with  pairs consisting of the column $A_{\_l}$ and row $B_{l\_}$
or, equivalently, with  rectangles $A_{\_l}\circ B_{l\_}$.
Furthermore, the objective to compute $A$ and $B$ with no/small error $E(I,A\circ B)$ may 
be rephrased as the goal to compute from $I$ a set of rectangles that exactly/approximately
cover~$I$.

%\begin{remark}\label{rem:decomp-cover}
%Geometrically, a $\max$-superposition of $J_1,\dots,J_k$ may be viewed as a covering.
%Each rectangle $J_l$ covers a set of entries, namely the set $\{\tu{i,j}\mid (J_l)_{ij}=1\}$.
%A $\max$-superposition of $J_1,\dots,J_k$ is the Boolean matrix in which the entry $\tu{i,j}$ contains $1$
%if and only if $\tu{i,j}$ is covered by at least one of the rectangles $J_1,\dots,J_k$.
%The objective to compute $A$ and $B$ with no/small error $E(I,A\circ B)$ may thus
%be rephrased as the goal to compute from $I$ a set of rectangles that exactly/approximately
%cover $I$.
%\end{remark}

Clearly, $E$ as defined by (\ref{eqn:error}) may be seen as being a sum of two components,
$E_u$ corresponding to $1$s in $I$ that are 0s in $A\circ B$ (``uncovered'')
and $E_o$ corresponding to $0$s in $I$ that are $1$s in 
$A\circ B$ (``overcovered''):
%\begin{equation}\label{eqn:error}
%  E(I,A\circ B)=\textstyle
%  E_u(I,A\circ B) + E_o(I,A\circ B),  \quad\text{where}
%\end{equation}
\begin{eqnarray*}%\label{eqn:error}
  &&E(I,A\circ B)=\textstyle
  E_u(I,A\circ B) + E_o(I,A\circ B),  \ \text{where}\\
  \nonumber
  &&\qquad E_u(I,A\circ B)=\textstyle
  |\{\tu{i,j} \,;\, I_{ij}=1, (A\circ B)_{ij}=0 \}|,
\\
  \nonumber
  &&\qquad   E_o(I,A\circ B)=\textstyle
  |\{\tu{i,j} \,;\, I_{ij}=0, (A\circ B)_{ij}=1 \}|.
\end{eqnarray*}
%
%\paragraph{From-Below Factorizations}
  Even though $E_u$ and $E_o$ look symmetric, they have
  a highly non-symmetric role in BMF.
Note that these two components are implicitly used in \Asso\ algorithm
\citep{MiMiGiDaMa:TDBP} and are treated non-symmetrically by function
\textit{cover} using two different weights.
The non-symmetry is seen from the following
observation which says that as we add new factors to already established ones
(i.e., add columns and rows to $A$ and $B$, respectively), 
$E_u$ may only decrease while $E_o$ may only increase. 
This property is easy to see using Observation \ref{thm:rec}.

\begin{observation}\label{thm:E}
%  Let $I\in\{0,1\}^{n\times m}$, $A\in\{0,1\}^{n\times k}$, and $B\in\{0,1\}^{k\times m}$.
  Let $A'\in\{0,1\}^{n\times (k+1)}$ and $B'\in\{0,1\}^{(k+1)\times m}$ result 
  by adding to $A$ and $B$ a single column and row, respectively.
  Then

  %\[
\centerline{
    $E_u(I,A'\circ B')\leq E_u(I,A\circ B) $
    and 
    %  \mbox{ and }
     $E_o(I,A'\circ B')\geq E_o(I,A\circ B) $.}
  %\]
\end{observation}
%\begin{proof}
%\end{proof}

The importance of Observation \ref{thm:E} derives from the following consideration.
Due to the provable hardness of the BMF related problems, such as DBP or AFP, 
it seems reasonable to assume that conceivable algorithms follow the logic of Observation~\ref{thm:E}
in that they output one factor after another.
%RR: (clearly, one may consider improvements
%such as backtracking but this does not eliminate the substance of this argument). 
This is indeed the case of the main existing algorithms  %\citep{BeVy:Dof,GeGoMi:Td,MiMiGiDaMa:TDBP}
discussed below.
With such algorithms, Observation \ref{thm:E} provides a warning. Namely, we should be 
careful with %M2: preklep committing
committing  $E_o$ error because $E_o$ 
never  decreases by adding further factors.

The most extreme strategy is not to commit $E_o$ error at all,
i.e. add the %M: preklep constraint
constraint $E_o(I,A\circ B)=0$.
As the requirement $E_o(I,A\circ B)=0$ is equivalent to $A\circ B\leq I$,
we call a BMF algorithm producing results with zero $E_o$ 
a \textit{from-below factorization algorithm} and say that $A$ and $B$
provide a \textit{from-below approximation} of $I$.
The restriction to from-below factorizations means that we exploit
only a restricted class of factorizations.
Surprisingly however, we show that such restriction leads to very good BMF
algorithms which outperform the available algorithms
producing the general factorizations, i.e. algorithms %M: preklep committing
committing $E_o$ error.
%, such as \Asso \citep{MiMiGiDaMa:TDBP} which seems to be the currently best of such algorithms. 
%
A further advantageous feature of the from-below approximations 
is the fact that they are amenable to theoretical analysis in terms of  closure and 
order-theoretic structures, as demonstrated below.

%\paragraph{Associated Closure Structures and Optimal From-Below Factorizations}
To every Boolean matrix $I\in \{0,1\}^{n\times m}$, one my associate a pair 
$\tu{{}^{\upts_I}, {}^{\downts_I}}$ (denoted also $\tu{{}^\upts, {}^\downts}$) of
operators assigning to  sets 
%R4: pridany mezery \quad
\[\text{$C\subseteq X=\{1,\dots,n\}$ \quad and \quad $D\subseteq  Y=\{1,\dots,m\}$}
\]
the sets
%RR:
%$C^\upts\subseteq Y$ and
%$D^\downts\subseteq X$ defined by
%\begin{eqnarray}
%  C^\upts = \{ j\in Y \mid \mbox{for each } i\in C: I_{ij}=1\},\\
%  D^\downts = \{ i\in X \mid \mbox{for each } j\in D: I_{ij}=1\}.
%\end{eqnarray}
\begin{eqnarray*}
  C^{\upts_I} = \{ j\in Y \mid \forall i\in C: I_{ij}=1\} \text{ and }
 D^{\downts_I} = \{ i\in X \mid \forall j\in D: I_{ij}=1\}.
\end{eqnarray*}
That is, $C^\upts$ is the set of all attributes (columns) shared by all
objects (rows) in $C$ and $D^\downts$ is the set of all objects
sharing all attributes in $D$.
The set
\[
    \mathcal{B}(I) = \{ \tu{C,D} \mid C\subseteq X, D\subseteq Y,
     C^\upts=D, D^\downts=C\}
\]
is called the \emph{concept lattice} of $I$, i.e. it is the set of all 
$\tu{{}^\upts, {}^\downts}$-closed pairs
$\tu{C,D}$,  called the \emph{formal concepts} of $I$, with $C$ and $D$  
called the extent and the intent. 
The set $\mathcal{B}(I)$ equipped with the partial order $\leq$
(modeling the subconcept-superconcept hierarchy) defined by
%\[
%    \tu{C_1,D_1}\leq \tu{C_2,D_2} \ \mbox{ if{}f } \ 
%     C_1\subseteq C_2 \ (\mbox{if{}f } D_1\supseteq D_2)
%\]
    $\tu{C_1,D_1}\leq \tu{C_2,D_2}$ if{}f 
     $C_1\subseteq C_2$ if{}f  $D_1\supseteq D_2$
forms indeed a complete lattice.
Concept lattices are utilized in 
formal concept analysis (FCA); we refer to \citep{DaPr:ILO,GaWi:FCA} for
details.
The pair $\tu{{}^\upts, {}^\downts}$ forms a Galois connection
between $X$ and $Y$ 
and
the compound mappings $^{\upts\downts}$ and ${}^{\downts\upts}$
form closure operators in $X$ and $Y$, respectively \citep{GaWi:FCA}.
%Recall that being a closure operator in $X$ means that ${}^{\upts\downts}$ fulfills
%\begin{equation}
%  A\subseteq A^{\upts\downts}, \quad
%  A_1\subseteq A_2 \text{ implies } A_1^{\upts\downts}\subseteq A_2^{\upts\downts}, \quad
%  A^{\upts\downts}=A^{\upts\downts\upts\downts}
%\end{equation}
%for every $A,A_1,A_2\subseteq X$.
%RR:, respectively 
 A concept lattice may be visualized using 
a particularly labeled line diagram and carries useful
information about the data $I$ which we utilize in our paper.
Note also that several polynomial-time delay algorithms are available
for computing $\mathcal{B}(I)$ \citep{KuOb:Cpagcl}.

An important link between BMF and formal concepts consists in the following facts.
First, in view of Observation \ref{thm:rec}, rectangles contained in $I$ are the building blocks 
of decompositions of $I$. Clearly, most efficient are the rectangles that are maximal 
w.r.t. containment $\leq$ defined by (\ref{eqn:cont}). As is well known, maximal rectangles
contained in $I$ correspond to formal concepts in $\mathcal{B}(I)$ in that 
$J$ is a maximal rectangle in $I$ if and only if there exists a formal concept
$\tu{C,D}\in\mathcal{B}(I)$ such that $J_{ij}=1$ is equivalent to $i\in C$ and $j\in D$.
This link is utilized in \citep{BeVy:Dof}, in particular in two BMF algorithms 
which we use in our experimental comparison below.
Next, we  generalize a theorem from \citep{BeVy:Dof}  regarding 
exact decompositions to from-below approximations.
Given a set
${\cal F} = \{\tu{C_1,D_1},\dots,\tu{C_k,D_k}\} \subseteq {\cal B}(I)$
(with a fixed indexing of the formal concepts $\tu{C_l,D_l}$),
%of formal concepts of $I$ induces 
define the $n\times k$ and $k\times m$ 
Boolean matrices $A_{\cal F}$ and $B_{\cal F}$ by
\begin{eqnarray}\label{eqn:A}
   (A_{\cal F})_{il}=\left\{
   \begin{array}{cc}
     1 & \mbox{ if } i\in C_l,\\
     0 & \mbox{ if } i\not\in C_l,
   \end{array}
   \right.
%\end{eqnarray*}
%and
%\begin{eqnarray*}\label{eqn:B}
  %\quad\mbox{and}\quad
  \quad\text{and}\quad
   (B_{\cal F})_{lj}=\left\{
   \begin{array}{cc}
     1 & \mbox{ if } j\in D_l,\\
     0 & \mbox{ if } j\not\in D_l,
   \end{array}
   \right.
\end{eqnarray}
for $l=1,\dots,k$.
That is, the $l$th column and $l$th row of $A$ and $B$ are the characteristic vectors
of $C_l$ and $D_l$, respectively.

\begin{remark}\label{rem:MTP}
  The preceding paragraph and Observation \ref{thm:rec}
  make it easy to see that the Minimum Tiling Problem (MTP) considered in 
  \citep{GeGoMi:Td} is equivalent to the 
  problem of finding an exact decomposition of a Boolean matrix.
  Namely, a database of $n$ objects and $m$ items considered in 
  \citep{GeGoMi:Td} may be identified with a Boolean matrix $I\in\{0,1\}^{n\times m}$; 
    a tile in $I$ is
   a pair $\tu{C,D}$ where $C\subseteq\{1,\dots,n\}$ and $D\subseteq\{1,\dots,m\}$
   such that every object in $C$ has every item in $D$.
   Hence, tiles in $I$ may be identified with rectangles contained in $I$. Moreover, maximal 
   tiles in $I$ are just formal concepts of $I$.
   MTP consists in finding a smallest set of tiles that cover the whole database.
   It is now clear that every set $\mathcal{F}$ of tiles of $I$ may be identified with matrices 
   $A_{\cal F}$ and $B_{\cal F}$ as in (\ref{eqn:A}), that $A_{\cal F}\circ B_{\cal F}\leq I$
   (i.e. $\mathcal{F}$ provides a from-below approximation of $I$),
   and that $\mathcal{F}$ is a solution to MTP if{}f $A_{\cal F}$ and $B_{\cal F}$ present
  a solution to the AFP problem from Section \ref{sec:nbn} for $\varepsilon=0$.
   \cite{GeGoMi:Td} proposed an algorithm for the MTP which we examine below.
   Note also that the connection of tiling to BMF is not mentioned in \citep{GeGoMi:Td}.
\end{remark}

%R4: binary -> Boolean
\begin{theorem}
  \label{thm:fbof}
  Let $A\circ B\leq I$ for $n\times k$ and $k\times m$ Boolean matrices $A$
  and $B$. Then there exists a set ${\cal F}\subseteq {\cal B}(I)$ 
  of formal concepts of $I$ with
  \(
    |{\cal F}|\leq k
  \)
  such that for the $n\times |{\cal F}|$ and $|{\cal F}|\times m$ 
  Boolean matrices $A_{\cal F}$ and $B_{\cal F}$ we have
  \[
     A_{\cal F}\circ B_{\cal F}\leq I \mbox{ and }  
      E(I,A_{\cal F}\circ B_{\cal F}) \leq E(I,A\circ B).     
  \]
\end{theorem}
\begin{proof}
The proof follows the same logic as the one of \citep[Theorem 2]{BeVy:Dof} and we include it
for reader's convenience.
An informal argument: each of the $k$ rectangles corresponding to $A\circ B$ is a rectangle in $I$ and is contained
in a maximal rectangle (formal concept) in $I$. The set $\mathcal{F}$ of these formal concepts has at most
$k$ elements and covers at least as many entries of $I$ as those covered by $A\circ B$.
Formally, every 
rectangle 
$J_l=A_{\_l}\circ B_{l\_}$ is contained in $I$. 
According to Observation \ref{thm:rec}, 
$I=\max_{l=1}^k J_l$.
Now consider the sets $C_l=\{ i \mid A_{il}=1\}$ and $D_l=\{ j \mid B_{lj}=1 \}$.
Every  $\tu{C_l^{\up\down},C_l^\up}$ is a formal concept in $\mathcal{B}(I)$
(a well-known fact in FCA).
Moreover $C_l\subseteq C_l^{\up\down}$, since ${}^{\up\down}$ is a closure operator. 
As
 $I_{il}=1$ for every $i\in C_l$ and $j\in D_l$, it follows that
$D_l\subseteq C_l^\up$.
Now consider the set 
\[
   \mathcal{F}=\{\tu{C_1^{\up\down},C_1^\up},\dots,\tu{C_k^{\up\down},C_k^\up}  \}
   \subseteq \mathcal{B}(I)
\]
and the matrices  $A_{\cal F}$ and $B_{\cal F}$.
Clearly $\mathcal{F}$ contains at most $k$ elements (it may happen
$|\mathcal{F}|<k$).
It is easy to check that the rectangle corresponding to $\tu{C_l^{\up\down},C_l^\up}$,
i.e. the cross-product $(A_{\cal F})_{\_l}\circ (B_{\cal F})_{l\_}$, is contained in $I$ and,
due to the above observation, contains $J_l$. Hence,
\[
     A\circ B =\max_{l=1}^k J_l \leq \max_{l=1}^k (A_{\cal F})_{\_l}\circ (B_{\cal F})_{l\_}
  = A_{\cal F}\circ B_{\cal F} \leq I.
\]
It follows that $E(I,A_{\cal F}\circ B_{\cal F}) \leq E(I,A\circ B)$,
finishing the proof.
%\hfill$\qed$
\end{proof}

%In Section \ref{sec:eI}, we provide further theoretical insight into the 
%from-below approximations of $I$, derive new heuristics that are experimentally tested and utilized
%in a design of a new BMF algorithm in the subsequent sections.
%In Section \ref{sec:red} we provide an independent obrservation,
%namely a simple technique to reduce
%time for computing decompositions of $I$ by reducing $I$ to a possibly smaller
%matrix whose exact decompositions are the same as the exact decompositions
%of $I$.

Theorem \ref{thm:fbof} asserts that decompositions utilizing formal concepts as factors 
% of the form $A_{\cal F}\circ B_{\cal F}$
are the best as far as the from-below approximations are concerned.
%R4: veta vynechana
%That is to say, formal concepts (equivalently, maximal tiles)  are not just some particular kinds 
%of factors. Rather, they are the most efficient factors to which one naturally arrives
%through the geometric view of BMF as the problem of covering by rectangles.
Moreover, formal concepts are easy to interpret, which is a relevant aspect from a data analysis
viewpoint.

%-------------------------------------------------------------
\subsection{Intervals in $\mathcal{B}(I)$, Role of Entries in $I$, and the Essential Part of $I$} \label{sec:eI}
% All entries are equal but some are better: good and bad elements to cover/not to cover}
%\paragraph{The Role of Entries of $I$}

%R4: zmenen zacatek vety
As we show in this section,
 formal concepts and other closure structures associated to the Boolean matrix 
$I\in\{0,1\}^{n\times m}$, such as the concept lattice $\mathcal{B}(I)$, 
help us understand the geometry of BMF. 
%R4: vynechano: This claim is illustrated in the present section. 
In particular, we show that the concept lattice $\mathcal{B}(I)$ may help us differentiate the role
of entries of $I$ in decompositions---an issue not addressed in the existing literature---and that intervals 
in $\mathcal{B}(I)$ play a crucial role in this regard.

%
%Clearly, the basic distinction is between the entries $\tu{i,j}$ containing $0$ and those containing $1$,
%i.e. $I_{ij}=0$ and $I_{ij}=1$, since the entries with $1$ are those that need to be covered to obtain
%an (exact) decomposition of $I$. 
%A further differentiation of the role of entries with $1$, an issue not discussed in the previous papers on BMF,
%is examined below.
%
%The entries $\tu{i,j}$ containing $1$ are essential for decompositions in that, as follows from the previous section, 
%for a set $\mathcal{F}\subseteq\mathcal{B}(I)$, we have 
%$I=A_\mathcal{F}\circ B_\mathcal{F}$ if{}f every 

%Intervals in $\mathcal{B}(I)$ are the subsets of $\mathcal{B}(I)$ of the form
  For formal concepts $\tu{C_1,D_1},\tu{C_2,D_2}\in\mathcal{B}(I)$, consider the subset of $\mathcal{B}(I)$ of the form
\begin{equation}\label{eqn:interval}
      [\tu{C_1,D_1},\tu{C_2,D_2}] =  \{ \tu{E,F}\in \mathcal{B}(I) \mid \tu{C_1,D_1} \leq\tu{E,F} \leq\tu{C_2,D_2} \}.
\end{equation}
Such a subset is called the \emph{interval in $\mathcal{B}(I)$ bounded by}  $\tu{C_1,D_1}$ and $\tu{C_2,D_2}$.
%R4: vynechano: (from below and above, respectively).
Furthermore, for $C\subseteq X$ and $D\subseteq Y$, let
$\gamma(C)=\tu{C^{\upts\downts},C^\upts}$  and 
           $\mu(D)= \tu{D^{\downts},D^{\downts\upts}}$, i.e. 
$\gamma(C)$ and $\mu(D)$ are the least formal concept in $\mathcal{B}(I)$ whose extent includes $C$ and the greatest
one whose intent includes $D$. Let us use $\gamma(i)$ and $\mu(j)$ instead of 
$\gamma(\{i\})$ and $\mu(\{j\})$ for row $i\in X$ and column $j\in Y$.
Denote
\begin{equation}\label{eqn:ICD}
      \mathcal{I}_{C,D} = [\gamma(C), \mu(D)].
\end{equation}
Clearly, every interval in $\mathcal{B}(I)$ is of the form (\ref{eqn:ICD}).
Of particular importance are the intervals of the form
\begin{equation}\nonumber
   \mathcal{I}_{ij} = [\gamma(i),\mu(j)].
\end{equation}
%for $i\in X$ and $j\in Y$.
%
%We now examine the role of the entries of $I$ containing $1$ in BMF,
%an issue not addressed by the existing algorithms.
%It turns out that a crucial information is carried by 
%certain intervals in the concept lattice $\mathcal{B}(I)$ of $I$.
%Namely, for an entry $\tu{i,j}$,  consider the formal concepts 
% $\gamma(i)$ and $\mu(j)$ induced by object $i$
%and attribute $j$, defined by
%\begin{eqnarray*}
%   \gamma(i) = \tu{ \{i\}^{\upts\downts}, \{i\}^{\upts} } \quad\mbox{ and }\quad
%   \mu(j) = \tu{ \{j\}^{\downts}, \{j\}^{\downts\upts} },
%\end{eqnarray*}
%and the interval
%\begin{equation}
%   \mathcal{I}_{ij} = [\gamma(i),\mu(j)]
%\end{equation}
%in $\mathcal{B}(I)$.
%Before we examine the role of intervals $\mathcal{I}_{ij}$, we introduce somewhat
%more general intervals and establish their basic properties utilized in 
%Section~\ref{sec:a}. For $C\subseteq X$ and $D\subseteq Y$, denote by 
%$\mathcal{I}_{C,D}$ the interval 
%\[
%      \mathcal{I}_{C,D} = [\gamma(C), \mu(D)] 
%\]
%in $\mathcal{B}(I)$
%where  $\gamma(C)=\tu{C^{\upts\downts},C^\upts}$ and 
%           $\mu(D)= \tu{D^{\downts},D^{\downts\upts}}$,
%i.e. the set 
%\begin{equation}\nonumber
%   [ \gamma(C), \mu(D)] = 
%       \{ \tu{E,F}\in \mathcal{B}(I) \mid \gamma(C) \leq\tu{E,F} \leq\mu(D) \}
%\end{equation}
%of all 
%%RR: formal 
%concepts bounded by the concepts $\gamma(C)$ and $\mu(D)$
% induced by $C$ and $D$, respectively.
%Clearly, $\mathcal{I}_{ij}$ is a special case of $\mathcal{I}_{C,D}$ for $C=\{i\}$ and $D=\{j\}$.
%
The following lemma describes the crucial properties for understanding
the role of intervals in from-below decompositions and is utilized
later in this section as well as in proving correctness of our new algorithm.

\begin{lemma} \label{thm:int}
 \begin{itemize}
   \item[\emph{(a)}]  $\mathcal{I}_{C,D}$ is non-empty if and only if
   $C\times D\subseteq I$, i.e. if $I_{ij}=1$ for every $i\in C$ and $j\in D$.
%   $I$ contains the rectangle given by $\tu{C,D}$, i.e. 
%   if $I_{ij}=1$ for every $i\in C$ and $j\in D$.
%   $C\times D\subseteq I$.
    In particular,  $\mathcal{I}_{ij}$ is non-empty if and only if $I_{ij}=1$.
   \item[\emph{(b)}] $\mathcal{I}_{C,D}\!=\!
      \{\tu{E,F}\in\mathcal{B}(I) \,|\, C\subseteq E, D\subseteq F\}\!=\!
      \{\tu{E,F}\in\mathcal{B}(I) \,|\, C^{\upts\downts}\subseteq E, D^{\downts\upts}\subseteq F\}$. 
     In particular, $\mathcal{I}_{ij}$ is the set of all concepts  that cover~$\tu{i,j}$.
   \item[\emph{(c)}] If $(A_{\cal F}\circ B_{\cal F})_{ij}=1$ %for  $\mathcal{F}\subseteq\mathcal{B}(I)$
                  then $\mathcal{F}$ contains at least one concept in $\mathcal{I}_{ij}$.
 \end{itemize}
\end{lemma}
\begin{proof}
 (a) As $\mathcal{I}_{C,D}\not=\emptyset$ if{}f $\gamma(C)\leq \mu(D)$,
   we need to check that $\gamma(C)\leq \mu(D)$ is equivalent to $C\times D\subseteq I$. 
   Using basic properties of Galois connections \citep{GaWi:FCA}, we get
   $\gamma(C)\leq \mu(D)$ if{}f $C^{\upts\downts}\subseteq D^\downts$
   if{}f $D\subseteq C^{\upts}$ if{}f for every $y\in D$ we have 
   $I_{ij}=1$ for every $i\in C$, i.e. if{}f $C\times D\subseteq I$.

  (b) We have $\tu{E,F}\in \mathcal{I}_{C,D}$ if{}f 
       $\gamma(C)\leq \tu{E,F}\leq \mu(D)$ if{}f
      $C^{\upts\downts}\subseteq E$ and $D^{\downts\upts}\subseteq F$.
      Now, since $C^{\upts\downts}$ is the least extent (${}^{\upts\downts}$-closed
      set of objects) containing $C$, $C^{\upts\downts}\subseteq E$ is equivalent
     to $C\subseteq E$; dually, $D^{\downts\upts}\subseteq F$ is equivalent to $D\subseteq F$,
     proving (b).

 (c) Due to Observation \ref{thm:rec} and (\ref{eqn:A}),
  $(A_{\cal F}\circ B_{\cal F})_{ij}=1$ means that there exists $\tu{C,D}\in\mathcal{F}$
   covering $\tu{i,j}$, whence (b) implies $\tu{C,D}\in \mathcal{I}_{ij}$.
%   \hfill$\qed$
\end{proof}

\begin{remark}\label{rem:mark}
Interestingly, our problem may be reformulated as a certain graph-marking problem.
Consider for a given matrix $I$ and $\varepsilon\geq 0$ the line diagram of the concept 
lattice $\mathcal{B}(I)$ \cite{DaPr:ILO,GaWi:FCA}, i.e. 
a labeled Hasse diagram in which the nodes represent formal concepts of $\mathcal{B}(I)$ and the nodes 
representing concepts $\gamma(i)$ and $\mu(j)$ are %MM: preklep labelled -> labeled
labeled by ``$i$'' and ``$j$''
(the diagram is explained in Example \ref{ex:mark}).
Due to Lemma \ref{thm:int} (b) and (c), the problem to find a smallest set $\mathcal{F}$ 
of formal concepts for which $E(I,A_\mathcal{F}\circ B_\mathcal{F})\leq\varepsilon$ 
(hence in case $\varepsilon=0$ the problem to find a decomposition 
$I=A_\mathcal{F}\circ B_\mathcal{F}$ with a smallest possible $\mathcal{F}$) 
may be reformulated as the following graph-marking problem:
In the line diagram of $\mathcal{B}(I)$, mark the smallest number of nodes such that with a possible 
exception of $\varepsilon$ cases, 
every non-empty interval $\mathcal{I}_{ij}$ (object $i$, attribute $j$) contains at least one marked node.
In other words, mark the smallest number of nodes such that with a possible exception of
$\varepsilon$ cases, if there is a path going upward from a node labeled ``$i$'' to a node labeled ``$j$'',
then one such path must contain a marked node.
%R4: geometric bez uvozovek
This geometric perspective is illustrated in Example \ref{ex:mark} and is used in what follows.
\end{remark}

\begin{example}\label{ex:mark}
  As an illustration, consider the following Boolean matrix $I$, representing objects
$1,\dots,6$ (rows) 
and attributes $a,\dots,e$ (columns). 
\begin{center}
\begin{minipage}[t]{0.35\columnwidth}%
\vspace{-5.4cm}
\centerline{
$
  \left( \begin{array}{ccccc}
   {1}  &   {1}   &  0   &  1  &   0\\
   1  &   0   &  0   &  1  &   {1}\\
   0  &   1   &  {1}   &  0  &   0\\
   0  &   0   &  0   &  {1} &   0\\
   {1}  &   1   &  1   &  1  &   0\\
   {1}  &   1   &  0   &  0  &   1\\   
    \end{array} \right)
$}
\end{minipage}%
\begin{minipage}[t]{0.6\columnwidth}%
\centerline{
\includegraphics[scale=1]{example}
}
\end{minipage}
\end{center}
The right part shows the line diagram of the concept lattice $\mathcal{B}(I)$
\citep{GaWi:FCA}. That is, the nodes and lines represent the concepts and the partial
order $\leq$ of $\mathcal{B}(I)$.
Every node represents the formal concept whose extent and intent consist of
the objects and attributes  attached to the node. 
%RRR: \\MARTIN: tohle jsem zmenil, aby to sedelo k prikladu:\\
Thus, the middle node of the three below the top node represents the formal concept
$\tu{\{1,2,4,5\},\{d\}}$, while first node of the three below the top node represents $\tu{\{1,2,5,6\},\{a\}}$.
There is a line from the node representing 
%RRR: obracene
$\tu{\{5\},\{a,b,c,d\}}$ up to the one representing  $\tu{\{1,5\},\{a,b,d\}}$ because
the first node is a direct predecessor of the second one.
In general, $\tu{C_1,D_1}\leq \tu{C_2,D_2}$ if{}f there is a path going up from the node representing
$\tu{C_1,D_1}$ to the one representing $\tu{C_2,D_2}$.
Bold object and attribute names indicate object and attribute concepts.
For instance, $\tu{\{1,5\},\{a,b,e\}}$ is the object concept $\gamma(1)$ because $1$ appears in bold
as a label at the corresponding node. 
Similarly, $\mu(a)=\tu{\{1,2,5,6\},\{a\}}$ is an object concept, while $\tu{\{1,2,5\},\{a,d\}}$ is neither
object nor attribute concept.
Observe what is true in general, namely that the objects (attributes) attached to every node are just those
that appear in bold on some downward (upward) path leading from the node. 
Hence, one can remove all the object and attribute labels except for the bold ones without
any loss of information and obtain the so-called  reduced labeling.

%R4: uprava vety
We can easily see from the diagram that the interval $\mathcal{I}_{4a}$ is empty,
corresponding to $I_{4a}=0$, while
\begin{eqnarray*}
  \mathcal{I}_{1a}&=&\{
 \tu{\{1,5\},\{a,b,d\}},\tu{\{1,2,5\},\{a,d\}}, \tu{\{1,5,6\},\{a,b\}},\\ 
    && \ \  \tu{\{1,2,5,6\},\{a\}}
\},
\end{eqnarray*}
 corresponding to $I_{1a}=1$.
In view of Lemma \ref{thm:int} and Remark \ref{rem:mark}, 
the set
%MARTIN: nevim zda tady byla zamerne neoptimalni faktorizace, tak jsem zase vlozil neoptimalni faktorizaci 
\begin{eqnarray*}
    \mathcal{F} &=&\{
         \tu{\{1,5\},\{a,b,d\}},\tu{\{1,2,4,5\},\{d\}}, \tu{\{2,6\},\{a,e\}},\tu{\{3,5\},\{b,c\}},\\
   && \ \  \tu{\{6\},\{a,b,e\}}
   \}
\end{eqnarray*}
is a set of factor concepts
%R4: of $I$
 of $I$, i.e. $I=A_\mathcal{F}\circ B_\mathcal{F}$, because
if we mark the nodes representing the concepts in $\mathcal{F}$, every non-empty interval
$\mathcal{I}_{ij}$ contains a marked node.
%DROBNA VADA NA KRASE: concepty v F jsou prave objektove koncepty
\end{example}

Clearly, a basic distinction may be made between the entries $\tu{i,j}$ of $I$ containing $0$ and those containing $1$,
i.e. between $I_{ij}=0$ and $I_{ij}=1$. Namely, the entries with $1$ are those that need to be covered by factors to obtain
an exact decomposition of $I$. 
In this sense, the entries of $I$ containing $1$ are sufficient.
Some of  these sufficient entries may, however, still be omitted and yet, the coverage of the remaining
entries still guarantees a decomposition of $I$, resulting in further differentiation
of the entries.
In general, with a prescribed precision $\varepsilon\geq 0$ of decomposition,
the differentiation %of the entries with $1$ according to their role in decompositions 
may be based on the following question. 
Is it possible to identify a matrix $J\leq I$ with a small number of $1$s
 with the property?:
\begin{equation}\label{eqn:ess_cov} 
  \text{for any $\mathcal{F}\subseteq\mathcal{B}(I)$, if 
%$A_{\cal F}\circ B_{\cal F}$ covers all the 1s of $S_\varepsilon$
%then the distance of $A_{\cal F}\circ B_{\cal F}$ and $I$ is at most $\varepsilon$, i.e.
$J\leq A_{\cal F}\circ B_{\cal F}$ then
$E(I,A_{\cal F}\circ B_{\cal F})\leq\varepsilon$
}
\end{equation}
Note that (\ref{eqn:cont}) says that the coverage of all $1$s in $J$ guarantees
the coverage of all $1$s in $I$ with a possible exception of $\varepsilon$ cases.
%R4: uprava
A Boolean matrix $J\leq I$ %MM: preklep satysfying -> satisfyng
satisfying (\ref{eqn:ess_cov}) that is minimal w.r.t.
 $\leq$ (i.e. the partial order defined by (\ref{eqn:cont})) is called \emph{$\varepsilon$-essential}
for $I$.
%In other words, 
%%RR: the entries with 1s of such $J$ may be seen as the entries
%%whose 
%the coverage of $1$s in $J$ by any $\mathcal{F}\subseteq \mathcal{B}(I)$ is sufficient
%for the error of $A_{\cal F}\circ B_{\cal F}$ to be at most $\varepsilon$.
In what follows, we restrict to $\varepsilon=0$ and call $0$-essential matrices simply 
\emph{essential}, or essential parts of $I$. We describe the essential parts 
and utilize them in a new decomposition algorithm. We shall demonstrate that essential
matrices prove useful  in computing exact decompositons as well as from-below approximations
of Boolean matrices.

%We call any matrix $J$ satisfying the above properties an \textit{essential}
%part (for exact decomposition) of $I$.
%RR:
 For $I\in\{0,1\}^{n\times m}$ denote by $\ess(I)$ the $n\times m$ Boolean matrix given by 
\[
    (\ess(I))_{ij} = 1 \mbox{ if{}f } \mathcal{I}_{ij} 
     \mbox{ is non-empty and minimal w.r.t. $\subseteq$,}
\]
where $\subseteq$ denotes set inclusion. Note that
\begin{eqnarray}
    \label{eqn:ord-int}
     \mathcal{I}_{ij}\! \subseteq \mathcal{I}_{i'j'} \mbox{ if{}f }
      \gamma(i')\!\leq\!\gamma(i) \,\mbox{and}\, \mu(j)\!\leq\!\mu(j')
%  \\    \nonumber
     \mbox{ if{}f } \{i\}^\upts\!\subseteq\! \{i'\}^\upts \,\mbox{and}\,
      \{j\}^\downts\!\subseteq\! \{j'\}^\downts
\end{eqnarray}
and that a non-empty $\mathcal{I}_{ij}$ is minimal w.r.t. $\subseteq$ if
it is not contained in any other $\mathcal{I}_{i'j'}$, i.e.
 $\mathcal{I}_{ij}=\mathcal{I}_{i'j'}$ whenever
$\mathcal{I}_{i'j'}\subseteq\mathcal{I}_{ij}$ for every $i',j'$.
The next theorem asserts that $\ess(I)$ is
%R4: the least -> a unique
a unique essential
part of $I$.
The theorem concerns \emph{clarified} matrices 
by which we mean
matrices with no identical rows and columns.
Clarification,
 i.e. removal of duplicit rows and columns,
 is a simple and useful preprocessing because it removes redundant information.
Moreover, it is easy to see that the decompositions of $I$ and its clarified $I'$
are in one-to-one correspondence.
In fact, as is readily seen from the proof, 
the sufficiency of $\ess(I)$ holds for general matrices;
and the assumption of clarification is used to prove uniqueness of $\ess(I)$.

\begin{theorem}
  \label{thm:ess}
   $\ess(I)$ is 
%R4: an -> a unique
   a unique  essential part of $I$, for every clarified $I$.
%  \emph{(a)}
%  $\ess(I)$ is an essential part of $I$.
%
%  \emph{(b)}
%  If $J$ is an essential part of a reduced $I$ then $\ess(I)\leq J$.
%
\end{theorem}
\begin{proof}
We need to show  (a) $\ess(I)\leq I$ and that for any ${\cal F}\subseteq {\cal B}(I)$,
  if $\ess(I)\leq A_{\cal F}\circ B_{\cal F}$ then 
  %$E(I,A_{\cal F}\circ B_{\cal F})=0$, i.e.
  $I=A_{\cal F}\circ B_{\cal F}$;
  and (b) if $J$ satisfies (a) then $\ess(I)\leq J$.

  (a)
  $\ess(I)\leq I$ follows from the definition of $\ess(I)$ and Lemma \ref{thm:int} (a).
  Let $\ess(I)\leq A_{\cal F}\circ B_{\cal F}$ and assume by contradiction that there exists
  $\tu{i,j}$ for which $I_{ij}=1$ and $(A_{\cal F}\circ B_{\cal F})_{ij}=0$.
  Consider any minimal interval $\mathcal{I}_{i'j'}\subseteq\mathcal{I}_{ij}$
  (at least one exists). By definition of $\ess(I)$, 
  $(\ess(I))_{i'j'}=1$, whence also $(A_{\cal F}\circ B_{\cal F})_{i'j'}=1$, from which it follows
  that there exists $\tu{C,D}\in\mathcal{F}$ which covers $\tu{i',j'}$.
  Due to Lemma \ref{thm:int} (b), $\tu{C,D}\in \mathcal{I}_{i'j'}$ whence
  Lemma \ref{thm:int} (b) and $\mathcal{I}_{i'j'}\subseteq\mathcal{I}_{ij}$ yield that
  $\tu{C,D}$ covers $\tu{i,j}$ and thus $(A_{\cal F}\circ B_{\cal F})_{ij}=1$,
  contradicting  the assumption. 
  
   (b) By contradiction, assume that  there exists
  $\tu{i,j}$ for which $\ess(I)_{ij}=1$ and $J_{ij}=0$.
   Since $\ess(I)\leq I$, we have $I_{ij}=1$.
   To prove the assertion, it hence suffices to show that there exists a set 
   ${\cal F}\subseteq {\cal B}(I)$ for which $J\leq A_{\cal F}\circ B_{\cal F}$
   and yet $(A_{\cal F}\circ B_{\cal F})_{ij}=0$. 
   For this purpose, consider an arbitrary $\mathcal{G}\subseteq {\cal B}(I)$
   for which $J\leq A_{\cal G}\circ B_{\cal G}$ (clearly, such $\mathcal{G}$ exists).
   If $\mathcal{G}$ does not contain any concept from $\mathcal{I}_{ij}$, we are done
   by taking $\mathcal{F}=\mathcal{G}$, since then $(A_{\cal F}\circ B_{\cal F})_{ij}=0$
   by Lemma \ref{thm:int} (c).
   Otherwise, do the following for every $\tu{C,D}\in\mathcal{G}\cap \mathcal{I}_{ij}$:
   Remove $\tu{C,D}$ from $\mathcal{G}$ and add instead for every $\tu{i',j'}$ for which
   $J_{i'j'}=1$ some formal concept $\tu{C_{i'},D_{j'}}\in \mathcal{I}_{i'j'}-\mathcal{I}_{ij}$
   and denote the resulting set of concepts by $\mathcal{F}$.
   Observe that such $\tu{C_{i'},D_{j'}}$ always exists. Namely, since $J\leq I$,
   $J_{i'j'}=1$ implies $I_{i'j'}=1$ and hence due to Lemma  \ref{thm:int} (a),
   $\mathcal{I}_{i'j'}$ is nonempty. Now, 
   $\mathcal{I}_{i'j'}\not=\mathcal{I}_{ij}$, since otherwise  we 
    have $\gamma(i)=\gamma(i')$ and $\mu(j)=\mu(j')$, and since $I$ is clarified,
   this yields $i=i'$ and $j=j'$ which is impossible since we assumed $J_{ij}=0$.
   Since $\mathcal{I}_{ij}$ is minimal w.r.t $\subseteq$ and different from 
   $\mathcal{I}_{i'j'}$, we  get the existence of $\tu{C_{i'},D_{j'}}\in \mathcal{I}_{i'j'}-\mathcal{I}_{ij}$.
    Now,  Lemma \ref{thm:int} (c) implies that $J\leq A_{\cal F}\circ B_{\cal F}$
   and yet $(A_{\cal F}\circ B_{\cal F})_{ij}=0$, finishing the proof.
%   \hfill$\qed$
\end{proof}

Note that the part of the previous theorem claiming the sufficiency to cover the entries $\tu{i,j}$ 
with $\ess(I)_{ij}=1$ to obtain exact decomposition is noted in \cite[p. 130]{GaGl:Ofa}.
% where the 
%authors call such entries tight. 
Note also that an extension of the theorem to general,  non-clarified matrices 
is simple but we omit it. Let us only note that for non-clarified matrices, 
essential matrices are not unique (they are easy to describe in terms of $\ess(I)$
and the duplicit rows and columns).
%instead of a single essential matrix
%one would deal with several such matrices.
%in that the single, smallest essential matrix $\ess(I)$ gets replaced by several minimal matrices,
% but we omit this issue.
Denote the union of intervals $\mathcal{I}_{ij}$ corresponding to 1s in $\ess(I)$ by 
$\mathcal{B}_\mathcal{E}(I)$, i.e.
\begin{equation}\label{eqn:BE}
  \mathcal{B}_\mathcal{E}(I) = \bigcup \{\mathcal{I}_{ij} \mid (\mathcal{E}(I))_{ij}=1 \}.
\end{equation}

\begin{example}
Consider again the matrix $I$ of Example \ref{ex:mark} and its concept lattice $\mathcal{B}(I)$, this time with a reduced labeling. 
The underlined entries in $I$ are just the essential $1$s, i.e. those with $\ess(I)_{ij}=1$.
\begin{center}
\begin{minipage}[t]{0.35\columnwidth}%
\vspace{-5.4cm}
\centerline{
$
  \left( \begin{array}{ccccc}
   \eess{1}  &   \eess{1}   &  0   &  1  &   0\\
   1  &   0   &  0   &  1  &   \eess{1}\\
   0  &   1   &  \eess{1}   &  0  &   0\\
   0  &   0   &  0   &  \eess{1} &   0\\
   {1}  &   1   &  1   &  1  &   0\\
   {1}  &   \eess{1}   &  0   &  0  &   \eess{1}\\   
    \end{array} \right)
$}
\end{minipage}%
\begin{minipage}[t]{0.6\columnwidth}%
\centerline{
\includegraphics[scale=1]{example_ess}
}
\end{minipage}
\end{center}
For instance, while $I_{2a}=1$, we have $\ess(I)_{2a}=0$ since the interval
$\mathcal{I}_{2a}$ contains a different, smaller interval, namely $\mathcal{I}_{2e}$, whence $\mathcal{I}_{2a}$ is not minimal.

The bold part of the diagram corresponds to $\mathcal{B}_{\ess}(I)$, i.e. the union
of the six intervals $\mathcal{I}_{ij}$ for which $\ess(I)_{ij}=1$, cf. (\ref{eqn:BE}).
One can now easily see that the set
\[
    \mathcal{F} =\{
         \tu{\{1,5,6\},\{a,b\}},\tu{\{1,2,4,5\},\{d\}}, \tu{\{2,6\},\{a,e\}},\tu{\{3,5\},\{b,c\}}
   \}
\]
covers all the $\tu{i,j}$ with $\ess(I)_{ij}=1$. Namely, in the diagram of $\mathcal{B}(I)$,
this is equivalent to the fact that that each of the six bold intervals $\mathcal{I}_{ij}$ contains some formal concept 
in $\mathcal{F}$ (see also Remark \ref{rem:mark} and Example \ref{ex:mark}).
Due to Theorem \ref{thm:ess}, $\mathcal{F}$ covers all entries of $I$ containing $1$,
whence $I=A_\mathcal{F}\circ B_\mathcal{F}$, i.e. $\mathcal{F}$ is a set of factor concepts of $I$.
In particular, we have
\begin{equation}\nonumber
  A_\mathcal{F}= 
  \left( \begin{array}{cccc}
	1&1&0&0\\
	0&1&1&0\\
	0&0&0&1\\
	0&1&0&0\\
	1&1&0&1\\
	1&0&1&0\\					
    \end{array} \right)
  \quad\text{and}\quad 
  B_\mathcal{F}= 
  \left( \begin{array}{ccccc}
	1&1&0&0&0\\
	0&0&0&1&0\\
	1&0&0&0&1\\
	0&1&1&0&0\\					
    \end{array} \right)  
\end{equation}

%CHYBA V PRIKLADU:
%Koncept s prazdnym koleckem patri do $\mathcal{I}_{2e}$, tedy ma byt cerny. Vsechny jsou tedy cerne.
%Novy priklad by mel splnovat puvodni pozadavky, dale pak to, ze bool. rank matice $I$ je mensi nez
%$\min(m,n)$, jako je to v tomto priklade. Martine, predelej to tedy a prepis slovni popis (formulace pouzij
%z tohoto prikladu).
%Priklad bych vytvoril bud pomoci basic theorem tak, ze bych nakreslil svaz a k nemu pak nasel
%kontext, nebo nahodnym generovani kontextu cca 5x5, tak aby bool. rank byl <5.
\end{example}

An interesting property of $\ess(I)$ whose further elaboration we utilize
in the new decomposition algorithm in Section \ref{sec:a}  is contained in the following
theorem showing how factorizations of $I$ may be obtained from factorizations
of $\ess(I)$.

\begin{theorem}\label{thm:GFfact}
%    Let $I$ be a Boolean matrix.
    Let $\mathcal{G}\subseteq\mathcal{B}(\ess(I))$ be a set of factor concepts of
    $\ess(I)$, i.e. $\ess(I)=A_\mathcal{G}\circ B_\mathcal{G}$.
    Then every set $\mathcal{F}\subseteq\mathcal{B}(I)$  containing
   for each $\tu{C,D}\in \mathcal{G}$  at least one concept
   from $\mathcal{I}_{C,D}$  is a set of
   factor concepts of $I$, i.e. $I=A_\mathcal{F}\circ B_\mathcal{F}$.
\end{theorem}
\begin{proof}
  Let for $\tu{C,D}\in\mathcal{G}$ denote by $\tu{E,F}_{\tu{C,D}}$ a concept
  in $\mathcal{F}\cap\mathcal{I}_{C,D}$ which exists according to the assumption.
  Due to Lemma \ref{thm:int} (a), $C\subseteq E$ and $D\subseteq F$. 
  Since this this is true for every $\tu{C,D}\in\mathcal{G}$, we readily obtain
   $A_\mathcal{G}\circ B_\mathcal{G}\leq A_\mathcal{F}\circ B_\mathcal{F}$.
  The assumption $\ess(I)=A_\mathcal{G}\circ B_\mathcal{G}$ now yields
  $\ess(I)\leq A_\mathcal{F}\circ B_\mathcal{F}$. As $\ess(I)$ is an essential part of $I$,
  we get $I=A_\mathcal{F}\circ B_\mathcal{F}$, finishing the proof. 
%  Due to Theorem \ref{thm:ess} and the definition of an essential part of $I$,
%  it is sufficient to show that $\ess(I)\leq A_\mathcal{F}\circ B_\mathcal{F}$.
%   The claim now follows from the assumption 
%
%  \hfill$\qed$
\end{proof}

\begin{remark}
   Clearly, Theorem \ref{thm:GFfact} may be generalized to arbitrary factorizations 
   of $\ess(I)$. 
  Namely, 
    suppose $\ess(I)=A\circ B$ for some $A\in \{0,1\}^{n\times k}$ and $B\in\{0,1\}^{k\times m}$
    and let $C_l=\{i \mid A_{il}=1\}$, $D_l=\{j \mid B_{lj}=1\}$
    for each $l=1,\dots,k$.
    Then every set $\mathcal{F}\subseteq\mathcal{B}(I)$  containing at least one concept
    from $\mathcal{I}_{C_l,D_l}$ for each $l=1,\dots,k$, is a set of factor concepts of $I$.
   Namely, each $\tu{C_l,D_l}$ may be extended to a formal concept of $\ess(I)$, 
   the collection $\mathcal{G}$ of all such formal concepts satisfies 
   $\ess(I)=A_\mathcal{G}\circ B_\mathcal{G}$ and then Theorem \ref{thm:GFfact} applies.
\end{remark}

  Theorem \ref{thm:GFfact} implies that the rank of $\ess(I)$ provides an upper bound on the rank of $I$:

\begin{theorem}\label{thm:rankIessI}
  For every Boolean matrix $I$ we have
\(
    \mathrm{rank}_\mathrm{B}(I) \leq  \mathrm{rank}_\mathrm{B}(\ess(I)).
\)
\end{theorem}
\begin{proof}
  Let 
  $k=\mathrm{rank}_\mathrm{B}(\ess(I))$, let 
  $A\in \{0,1\}^{n\times k}$ and $B\in \{0,1\}^{k\times m}$ such that
  $\ess(I)=A\circ B$.
  Consider any $\mathcal{F}\subseteq \mathcal{B}(I)$ containing for each 
  $\tu{C,D}\in\mathcal{G}$ exactly one formal concept in the interval $\mathcal{I}_{C,D}$
  of $\mathcal{B}(I)$. 
  Note that such $\mathcal{F}$ exists
  since for every $\tu{C,D}\in\mathcal{G}$ we have 
    $C\times D\subseteq \ess(I)$ because  $\tu{C,D}$ is a formal concept of $\ess(I)$,
  and hence $\ess(I)\leq I$ implies $C\times D\subseteq I$.
  Due to Lemma \ref{thm:int}  (a), $\mathcal{I}_{C,D}$ is a non-empty interval in
  $\mathcal{B}(I)$.
  According to Theorem \ref{thm:GFfact}, $I=A_\mathcal{F}\circ B_\mathcal{F}$.
  Now, clearly, 
  $\mathrm{rank}_\mathrm{B}(I)\leq |\mathcal{F}|\leq |\mathcal{G}|=
   \mathrm{rank}_\mathrm{B}(\ess(I))$, finishing the proof.
%  \hfill$\qed$
\end{proof}

\begin{remark}
  The estimation by Theorem~\ref{thm:rankIessI} is not tight. Namely,
  as one may check, 
\[
\text{for }
I = 
\left( \begin{array}{ccccc}
1 0 1 1 1\\
0 1 1 0 1\\
0 1 0 0 1\\
1 0 1 1 0\\
    \end{array} \right) = 
\left( \begin{array}{ccc}
1 1 0 \\
0 1 1 \\
0 0 1 \\
1 0 0 \\
    \end{array} \right)
\circ
\left( \begin{array}{ccccc}
1 0 1 1 0\\
0 0 1 0 1\\
0 1 0 0 1\\
    \end{array} \right),
\text{ we have }
     \ess(I) = 
\left( \begin{array}{ccccc}
0 0 0 0 1\\
0 0 1 0 0\\
0 1 0 0 0\\
1 0 0 1 0\\
    \end{array} \right).
\]
While we see that $\mathrm{rank}_\mathrm{B}(I)\leq 3$ (in fact, the rank equals $3$), one
may easily check that $\mathrm{rank}_\mathrm{B}(\ess(I))=4$.
\end{remark}

Another issue %M: preklep, connected
connected to $\ess(I)$ is the following.
Due to Theorem 2 of \citep{BeVy:Dof}, optimal exact decompositions may be obtained 
by using as factors the formal concepts of $\mathcal{B}(I)$. The next theorem
%R4: vynechano improves this result in that it 
shows that even more restricted formal concepts are sufficient, namely
those in $\mathcal{B}_\mathcal{E}(I)$.

\begin{theorem}\label{thm:fcaofEss}
  For every $I$ there exists $\mathcal{F}\subseteq\mathcal{B}_{\ess}(I)$
  with $|\mathcal{F}|=\mathrm{rank}_\mathrm{B}(I)$ 
  for which
  $A_\mathcal{F}\circ B_\mathcal{F}=I$.
% (i.e. $\mathcal{F}$ is a set of factor concepts)
%  and $|\mathcal{F}|=\mathrm{rank}_\mathrm{B}(I)$ (Schein rank, the least number of factors).
%  (Kratce: There exist an optimal factor set $\mathcal{F}\subseteq\mathit{Ess}(\mathcal{B}(X,Y,I))$.)
\end{theorem}
\begin{proof}
  Due to Theorem 2 of \citep{BeVy:Dof}, there exists $\mathcal{F}\subseteq\mathcal{B}(I)$ 
  with $|\mathcal{F}|=\mathrm{rank}_\mathrm{B}(I)$ for which
  $A_\mathcal{F}\circ B_\mathcal{F}=I$.
  It is sufficient to  show that $\mathcal{F}\subseteq\mathcal{B}_{\ess}(I)$.
  Suppose by contradiction that there exists 
  $\tu{C,D}\in\mathcal{F}-\mathcal{B}_{\ess}(I)$, i.e. 
  $\tu{C,D}$ does not belong to any minimal $\mathcal{I}_{i,j}$. 
  Since $\mathcal{F}$ covers $I$, and hence also $\mathit{Ess}(I)$, 
  for every $\tu{i',j'}$ with $\ess(I)_{i'j'}=1$, there exists a formal concept
  $\tu{C',D'}\in\mathcal{F}$ which covers $\tu{i',j'}$. By Lemma \ref{thm:int} (b),
  $\tu{C',D'}\in \mathcal{I}_{i',j'}$ and hence $\tu{C',D'}\not=\tu{C,D}$. 
  For $\mathcal{F}'=\mathcal{F}-\{\tu{C,D}\}$ we thus have 
  $A_{\mathcal{F}'}\circ B_{\mathcal{F}'}\geq \ess(I)$ hence 
   $A_{\mathcal{F}'}\circ B_{\mathcal{F}'}=I$ by Theorem \ref{thm:ess}.
  We obtained a set $\mathcal{F}'$ of factors of $I$ with $|\mathcal{F}'|=|\mathcal{F}|-1<
  \mathrm{rank}_\mathrm{B}(I)$, a contradiction. 
%  \hfill$\qed$
\end{proof}

\begin{remark}
  Theorem \ref{thm:fbof}, dealing with approximate from-below factorizations,  
  and Theorem \ref{thm:fcaofEss}, dealing with a stronger restriction of the search space
  for optimal factorizations,
  provide two different improvements of Theorem 2 of \citep{BeVy:Dof}.
  The following example shows that the natural common generalization of Theorems \ref{thm:fbof} and 
  \ref{thm:fcaofEss} does not hold. Namely, the generalization 
 would result from Theorem \ref{thm:fbof}  by replacing the condition 
 $\mathcal{F}\subseteq\mathcal{B}(I)$ by  $\mathcal{F}\subseteq\mathcal{B}_\mathcal{E}(I)$.
  Consider the following  %Boolean 
matrix $I$ and the %formal 
concept $\tu{C,D}\in\mathcal{B}(I)$:
  
 \[
  I=
  \left(
   \begin{array}{cccc}
      1 & \eess{1} & 0 & 0  \\
      1 & 0 & \eess{1} & 0  \\
      1 & 0 & 0 & \eess{1}
  \end{array}
  \right),
  \qquad
  \tu{C,D}=\{\{1,2,3\},\{1\}\}.
 \]
 For $\mathcal{G}=\{\tu{C,D}\}$ and $A=A_\mathcal{G}$ and 
  $B=B_\mathcal{G}$ we have $A\circ B\leq I$
 and $E(I,A\circ B)=3$. 
 Now, $\tu{C,D}\not\in\mathcal{B}_{\ess}(I)$ and there is no one-element subset
 $\mathcal{F}\subseteq \mathcal{B}_{\ess}(I)$ for which
 $E(I,A_\mathcal{F}\circ B_\mathcal{F})\leq E(A\circ B,I)$,
  i.e. the generalization does not hold.
 \end{remark}

The next lemma is easy to see and shows that $\ess(I)$ can be computed easily:
%RR: from $I$.

%M: pise se ze two following, ale podminky jsou  tri
%R4: opraveno
%------------------------------
\begin{lemma}\label{thm:EssComp}
  $\ess(I)_{ij}=1$ if an only if the following conditions are fulfilled:
  \begin{itemize}
    \item[\emph{(a)}] $I_{ij}=1$;
   
%RRR: u (b) a (c) doplneny zavorky
    \item[\emph{(b)}]
    for every $i'$ with $\{i'\}^\upts\subset\{i\}^\upts$ we have $I_{i'j}=0$
    (i.e., no $i'$ whose row is  contained in the row of $i$ contains $j$);
%RRR: (c) byly sipky nahoru, dal pak $I_{i'j}=0$
    \item[\emph{(c)}]
    for every $j'$ with $\{j'\}^\downts\subset\{j\}^\downts$ we have $I_{i'j}=0$
    (i.e., no $j'$ whose column is  contained in the column of $j$ contains $i$).
  \end{itemize}  
\end{lemma}
%\begin{proof}
%   $\{i'\}^\upts\subset\{i\}^\upts$ is equivalent to $\gamma(i) < \gamma(i')$
%   and $I_{i'j}=0$ is equivalent to $\gamma(i')\not\leq\mu(j)$.
%    Hence, (a) is means that for every row $i'$, if $\gamma(i) < \gamma(i')$ then $\gamma(i')\not\leq\mu(j)$.
%   Likewise, (b) means that for every column $j'$, if 
%   $\mu(j') < \mu(j)$ then $\gamma(i)\not\leq\mu(j')$.
%   Hence, (a) and (b) mean that $\mathcal{I}_{i,j} =[\gamma(i),\mu(j)]$ is minimal.
% \end{proof}

%%%%%%%%%%%%%%%%%%%%%%%%%%%%%%%%%%%%%%%%%%%%%%%%%%%%%%%
\section{\GreEss: A New BMF Algorithm}\label{sec:a}

The new BMF algorithm described in this section is primarily designed for AFP but
can also be used for DBP (see Section \ref{sec:p}).
The algorithm is based on the properties
of essential parts $\ess(I)$ of Boolean matrices $I$ established above.
%RR: in Section \ref{sec:fba}.
Two important features of essential parts we start with
are the following. First, since $\ess(I)$ represents
the entries of $I$ with $1$ whose coverage by arbitrary factors guarantees
% the coverage of all $1$s in $I$, and thus grarantees 
an exact decomposition of $I$ by these factors,
$\ess(I)$ provides us with information about where to focus in the search
for factors of $I$.
%in constructing factors for decomposition of $I$, i.e. in covering the $1$s in $I$.
Second, since the number $||\ess(I)||$ of $1$s in $\ess(I)$ tends to be significantly
smaller than the number $||I||$ of $1$s in $I$ (see Section \ref{sec:ac}),
covering the $1$s in $\ess(I)$ tends to be simpler than covering the $1$s in the original matrix $I$.
Note also that due to Lemma \ref{thm:EssComp}, % (\ref{eqn:ord-int}), 
$\ess(I)$ is computed easily.

In particular, we build upon the following idea. We intend to form a collection $\mathcal{G}$ of
possibly overlapping groups of essential $1$s, i.e. $1$s in $\ess(I)$.
These are 
considered as ``seeds'' for finding factors of $I$ in that each group $g\in \mathcal{G}$ 
is to be covered by a factor---a formal concept $\tu{C,D}\in\mathcal{B}(I)$
(taking formal concepts of $I$ as factors is optimal for AFP due to Theorem \ref{thm:fbof}).
Since each concept $\tu{C,D}\in\mathcal{B}(I)$ corresponds to a maximal rectangle in 
$I$, every group $g$ covered by $\tu{C,D}$ may be extended to a maximal rectangle
in $\ess(I)$, i.e. to a formal concept in $\mathcal{B}(\ess(I))$, which will still be covered
by $\tu{C,D}$.
Therefore, reasonable candidates for the groupings $\mathcal{G}$ are sets of formal 
concepts of $\ess(I)$, i.e. $\mathcal{G}\subseteq\mathcal{B}(\ess(I))$.

Considering $\mathcal{G}\subseteq\mathcal{B}(\ess(I))$ is a reasonable
strategy in view of Theorem \ref{thm:GFfact}, which suggests to compute a factor
set $\mathcal{G}$ of $\ess(I)$ and then compute from $\mathcal{G}$ a factor set 
$\mathcal{F}$ of $I$.
Our algorithm utilizes an improvement of this idea. Namely,  we compute 
a  set $\mathcal{G}\subseteq\mathcal{B}(\ess(I))$ 
which need not be a factorization of $\ess(I)$, i.e. may be smaller and satisfy
$A_\mathcal{G}\circ B_\mathcal{G}<\ess(I)$. 
Nevertheless,  $\mathcal{G}$ still has
the following important property:
If for each $\tu{C,D}\in\mathcal{G}$ we pick exactly one concept $c_{\tu{C,D}}$ in the 
interval $\mathcal{I}_{C,D}$ of $\mathcal{B}(I)$, then no matter how we pick,
the resulting $\mathcal{F}=\{c_{\tu{C,D}}\mid \tu{C,D}\in\mathcal{G}\}$
provides a factorization of $I$, i.e. $A_\mathcal{F}\circ B_\mathcal{F}=I$.
The pseudocode of our algorithm,  called \GreEss,  is described in Algorithm \ref{alg:GreEss} and %Algorithm
 \ref{alg:CC}.
\GreEss\ is designed for AFP,
i.e. it takes as its input a
%RR: Boolean 
matrix $I$ and $\varepsilon\geq 0$ and produces
a set $\mathcal{F}$ of formal concepts of $I$ for which
%and the induced) matrices
%$A_\mathcal{F}$ and $B_\mathcal{F}$ for which 
$||I-A_\mathcal{F}\circ B_\mathcal{F}||\leq \varepsilon$.
Hence, for $\varepsilon=0$ the algorithm produces an exact decomposition of $I$.

\begin{algorithm}
\small
\DontPrintSemicolon
\LinesNumbered
\KwIn{Boolean matrix $I$ and $\varepsilon\geq 0$}
\KwOut{set $\mathcal{F}$ of factors for which 
   $||I-A_\mathcal{F}\circ B_\mathcal{F}||\leq\varepsilon$}
\BlankLine
$\mathcal{G} \leftarrow \ComputeIntervals(I)$\;
$U \leftarrow \{\left<i,j\right> | I_{ij} = 1 \}$;
$\mathcal{F} \leftarrow \emptyset$ \;

%\BlankLine

\While{$|U| >\varepsilon$}
{
$s \leftarrow 0$\;
\ForEach{$\left<C,D \right> \in \mathcal{G}$}
{
% MISTO $J =  D^{\downarrow_I} \times C^{\uparrow_I}$ \;   DAVAM
$J \leftarrow   I\cap (D^{\down_I}\times C^{\up_I})$ \;  
 $F \leftarrow \emptyset$;
$s_{\tu{C,D}} \leftarrow 0$ \;
 \While{exists $j\in C^{\uparrow_I}-F$ s.t. %misto C_f jsem dal C
$|(F \cup \{j\})^{\downarrow_J} \times (F \cup \{j\})^{\downarrow_J \uparrow_J} \cap U|>
  s_{\tu{C,D}} $ }
 {
    select $j$ maximizing 
    $|(F \cup \{j\})^{\downarrow_J} \times (F \cup \{j\})^{\downarrow_J \uparrow_J} \cap U|$ \;
    $F \leftarrow (F \cup \{j\})^{\downarrow_J \uparrow_J}$
   $E \leftarrow (F \cup \{j\})^{\downarrow_J}$ \;
   $s_{\tu{C,D}}   \leftarrow |E \times F \cap U|$ \;
 }
 \If{$s_{\tu{C,D}}>s$}{
% UPRAVIT ZALOMENI, , aby tri prikazy prirazeni byly pod sebou
    $\tu{E',F'}\leftarrow\tu{E,F}$ \; 
    $\tu{C',D'}\leftarrow\tu{C,D}$ \; 
    $s\leftarrow s_{\tu{C,D}}$}   %TADY PRIDAN UPDATE S
}
\KwSty{add} $\left<E', F' \right>$ \KwSty{to} $\mathcal{F}$ \;
 \KwSty{remove}  $\tu{C',D'}$ \KwSty{from} $\mathcal{G}$ \;

$U\leftarrow U-E'\times F'$\;

}
\Return{$\mathcal{F}$}

\caption{\GreEss}
\label{alg:GreEss}
\end{algorithm}

\begin{algorithm}
\small
\DontPrintSemicolon
\LinesNumbered
\KwIn{Boolean matrix $I$}
\KwOut{Set $\mathcal{G}\subseteq\mathcal{B}(\ess(I))$}

$\ess \leftarrow \ess(I)$ 
$U \leftarrow \{\left< i, j\right> | \ess_{ij} = 1\}$ 
%RRR: pridano, MARTINE zkontroluj celkove cisla radku v popisu algoritmu
$\mathcal{G}\leftarrow\emptyset$

%\BlankLiner

\While{$U \neq \emptyset$}
{
$D \leftarrow \emptyset$;
$s \leftarrow 0$ \;

 \While{exists $j \notin d$ with
$|((D \cup \{j\})^{\downarrow_{\ess}})^{\uparrow_I \downarrow_I} \times ((D \cup \{j\})^{\downarrow_{\ess} \uparrow_{\ess}})^{\downarrow_I \uparrow_I} \cap U|>s$ }
 {
 select $j$ 
%R4: which maximize -> which maximizes
  which maximizes $|((D \cup \{j\})^{\downarrow_{\ess}})^{\uparrow_I \downarrow_I} \times ((D \cup \{j\})^{\downarrow_{\ess} \uparrow_{\ess}})^{\downarrow_I \uparrow_I} \cap U|$ \;
$D \leftarrow (D \cup \{j\})^{\downarrow_{\ess} \uparrow_{\ess}}$;
$C \leftarrow (D \cup \{j\})^{\downarrow_{\ess}} $ \;
$s  \leftarrow |C^{\uparrow_I \downarrow_I} \times D^{\downarrow_I \uparrow_I} \cap U|$ \;
 }
 
\KwSty{add} $\left<C,D \right>$ \KwSty{to} $\mathcal{G}$ \;

$U\leftarrow U-C^{\uparrow_I \downarrow_I} \times D^{\downarrow_I \uparrow_I}$

}
%R4: \Return{$C, D$}
\Return{$\mathcal{G}$}

\caption{\ComputeIntervals}
\label{alg:CC}
\end{algorithm}

We now provide a detailed description of this algorithm and justify its correctness.
We shall need the following lemma.
\begin{lemma}\label{thm:ICD-BJ}
  Let $C\times D\subseteq I$, i.e. $I_{ij}=1$ for every $i\in C$ and $j\in D$, let $J=I\cap (D^{\down_I}\times C^{\up_I})$.
  Then $\mathcal{I}_{C,D}=\mathcal{B}(D^{\down_I},C^{\up_I}, J)$.
\end{lemma}
\begin{proof}
  Let $\tu{E,F}\in \mathcal{I}_{C,D}$. Since the least and the greatest concepts in $\mathcal{I}_{C,D}$
  are $\tu{C^{\up_I\down_I},C^{\up_I}}$ and $\tu{D^{\down_I},D^{\down_I\up_I}}$, respectively, we have
  $E\subseteq D^{\down_I}$ and $F\subseteq C^{\up_I}$. Since $E^{\up_I}=F$ and $F^{\down_I}=E$ and since
  $J$ is the restriction of $I$ to $D^{\down_I}\times C^{\up_I}$, we clearly have 
  $E^{\up_J}=F$ and $F^{\down_J}=E$, establishing $\tu{E,F}\in\mathcal{B}(D^{\down_I},C^{\up_I}, J)$.
  
   Conversely, let $\tu{E,F}\in\mathcal{B}(D^{\down_I},C^{\up_I}, J)$.
   Clearly, $C\times C^{\up_I}\subseteq I$ and $D^{\down_I}\times D\subseteq I$.
   Since $C\times D\subseteq I$, we have $C\subseteq D^{\down_I}$ and $D\subseteq C^{\up_I}$.
   Therefore,  
   since $J$ is a restriction of $I$, $C\times C^{\up_I}\subseteq J$ and $D^{\down_I}\times D\subseteq J$.
   Since $F\subseteq C^{\up_I}$ and $E\subseteq D^{\down_I}$,
   we obtain $C\subseteq C^{\up_I\down_J}\subseteq F^{\down_J}=E$ and
   $D\subseteq D^{\down_I\up_J}\subseteq E^{\up_J}=F$.
   It remains to verify $E^{\up_I}=F$ and $F^{\down_I}=E$. 
   $E^{\up_I}\supseteq F$ follows directly from $E^{\up_J}=F$  and $J\leq I$. 
   On the other hand, if $j\in E^{\up_I}$ then since $C\subseteq E$ and hence $E^{\up_I}\subseteq C^{\up_I}$,
   we have $j\in C^{\up_I}$. This means that $j$ is an attribute of the context 
   $\tu{D^{\down_I},C^{\up_I}, J}$ nd since $J$ is a restriction of $I$, $j\in E^{\up_I}$ implies
   $j\in E^{\up_J}=F$, proving $E^{\up_I}\subseteq F$.
   $F^{\down_I}=E$ can be verified in a similar manner.
\end{proof}
\GreEss\ first calls \ComputeIntervals\ (described in detail below) which computes the aforementioned $\mathcal{G}$.
In the loop %M: chybne radky, ma byt 3--22
3--22, \GreEss\ is picking concepts $\tu{E',F'}$ from intervals
$\mathcal{I}_{C,D}$, for  $\tu{C,D}\in\mathcal{G}$, in such a way that at most one
concept is selected from every interval, until the size of the collection $U$ of uncovered
entries of $I$ is less than $\varepsilon$.
In %M: chybne radky, ma byt 5--18
5--18, the yet unused intervals $\mathcal{I}_{\tu{C,D}}$ are searched in a greedy
manner as follows. We start by the attribute concepts
$\gamma({j})\in\mathcal{I}_{C,D}$
and try to extend them greedily by adding attributes (loop 8--12).
Due to Lemma \ref{thm:ICD-BJ},
restricting to $J$ and using ${}^{\up_J}$ and ${}^{\down_J}$  guarantees that we do not leave $\mathcal{I}_{C,D}$.
Note that the greedy extension by attributes is inspired by \citep{BeVy:Dof}.
A possible extension of the so far best found $\tu{E,F}$ is accepted (l. 8) if 
the extended concept $\tu{ (F \cup \{j\})^{\downarrow_J},(F \cup \{j\})^{\downarrow_J \uparrow_J}}$
covers a larger number of entries of $U$ than the number $s_{\tu{C,D}}$ covered by $\tu{E,F}$. 
The best such concept of all the unused intervals, denoted $\tu{E',F'}$, is
then added to $\mathcal{F}$, the interval is marked as used
by removing $\tu{C',D'}$ from $\mathcal{G}$ %M: chybne radky, ma byt 20
(l. 20), and
the entries covered by $\tu{E',F'}$ are removed from $U$ %M: chybne radky, ma byt 21 
(l. 21).
\ComputeIntervals\ computes $\ess(I)$, computes and adds to $\mathcal{G}$
the concepts of $\mathcal{B}(\ess(I))$ computed in a greedy manner by
starting from attribute concepts and adding attributes similarly as above.
The difference, however, is that the entries of $U$ that are considered as covered by
a candidate $\tu{C,D}$ are those in
$C^{\uparrow_I \downarrow_I} \times D^{\downarrow_I \uparrow_I}$,
i.e. not only those of $C\times D$.
This is possible because the factors for $I$  are selected from the intervals $\mathcal{I}_{C,D}$
in \GreEss\ and due to Lemma \ref{thm:int} (b), every such factor covers
$C^{\uparrow_I \downarrow_I} \times D^{\downarrow_I \uparrow_I}$.
This is why $\mathcal{G}$ may not be  a factorization of $\ess(I)$,
i.e. may be smaller.
These considerations tell us:

%M: tady bych dopsal dukaz.
%R4: nechal bych to byt, pri prochazeni mi argumenty pripadaly jasne
\begin{theorem}
 \GreEss\ is correct and provides a from-below approximation~of~$I$.
% \GreEss computes a from-below approximation~of~$I$ that provides an approximate solution to 
% AFP, i.e. $E(I,A_\mathcal{F}\circ B_\mathcal{F})$.
\end{theorem}
\section{Experimental Evaluation}
\label{sec:e}

In this section, we provide an experimental evaluation of  \GreEss\ 
and its comparison with five main existing BMF algorithms, described in 
Section \ref{sec:ac}.
We use synthetic and real datasets in a scenario similar to that used in 
\citep{MiMiGiDaMa:TDBP} and other
%R4 papers. -> papers on BMF.
papers on BMF.
In addition, we provide statistical evaluation of the characteristics described in 
Section \ref{sec:fba}.

%-------------------------------------------------------------
\subsection{Algorithms %for Comparison 
and Datasets}\label{sec:ac}
%For experimental comparison with the techniques studied in our paper,
%we chose two available algorithms, namely the \Asso algorithm from
%\cite{MiMiGiDaMa:TDBP} and Algorithm 2 from \cite{BeVy:Dof}  which we denote
%A2 in this paper. Both are well-described in the literature, were tested
%experimentally on extensive collections of datasets and perform
%reasonably well.
%We omitted some other approaches mainly because htey are not described well
%in the literature or are not efficient (basically because they perform a kind of 
%a brute force search).

\paragraph{Algorithms}
We now describe the algorithms used in our comparison.
Further information and comments regarding these algorithms are provided
in the next parts of this section.

%\paragraph{Tiling}
\emph{\Tiling}
 is the algorithm proposed in \citep{GeGoMi:Td} for the Minimum Tiling Problem
(Remark \ref{rem:MTP}). The algorithm utilizes a modification of the $k$-LTM algorithm
proposed also in \citep{GeGoMi:Td}. To find a small tiling, the algorithm iteratively computes a 
 tile  that covers the largest number of still uncovered $1$s of all the tiles in $I$, until all
$1$s in $I$ are covered, implementing thus the well-known greedy set cover algorithm.
%RR: \citep{Hro:AHP}. 
It follows from Remark \ref{rem:MTP} that Tiling provides a from-below
factorization of $I$.

%\paragraph{\Asso} 
\emph{\Asso} 
\citep{MiMiGiDaMa:TDBP}, probably the most discussed
BMF algorithm in the data mining literature, 
first computes an $m\times m$ Boolean matrix $A$ in which
$A_{ij}=1$ if the confidence of rule $\{i\}\Rightarrow \{j\}$ exceeds
a parameter $\tau$. The rows of $A$ are then used as candidates
for the rows of the factor-attribute (i.e., basis vector) matrix.
The actual rows are selected using a greedy approach using 
function $\textit{cover}$ that rewards with weight $w^+$ the decrease
of error $E_u$ and penalizes with weight $w^-$ the increase 
of $E_o$ that is due to a given row of $A$.
%M: mezera za makrem, opraveno
%R4: uprava vety
\Asso\ is designed to solve the Discrete Basis Problem (Section \ref{sec:p}) and commits both 
types of errors, $E_u$ and $E_o$. 
It thus provides general factorizations which need not be
%M: 2 preklepym beow factoriations
from-below factorizations of $I$.

\emph{\Hyper} 
\cite{Xiea:Stdoh} 
produces from a Boolean matrix $I$ a set $\mathcal{F}$ of rectangles (called hyperrectangles by the authors) in $I$
that provide an exact decomposition of $I$. \Hyper\  attempts to find $\mathcal{F}$ with the smallest minimal cost. 
For cost, \Hyper\ employs the minimal description length (MDL) principle in 
that for a rectangle 
$\tu{C,D}$ ($C$ and $D$ being sets of rows and columns, respectively), its cost is
$\mathit{cost}(C,D)=|C|+|D|$ and the cost of  set $\mathcal{F}$ of rectangles is
the sum of the costs of the rectangles in $\mathcal{F}$. Analogously as in case of DBP and AFP,
constraints on error and the number of factors are %M: preklep considered
considered in modifications of the decomposition problem.
Hence, instead of the number of 
%R4: factor -> factors
factors themselves, the primary concern for \Hyper\ is the description
length of the factors. This makes the problem different from DBP and AFP. However, the fact that a small number of 
rectangles (factors) tends to have small description length and the claims in \cite{Xiea:Stdoh} that \Hyper\ provides a good
summarization of the data makes \Hyper\ 
%R4: vynechano certainly
 a relevant decomposition algorithm.

%RRR: popis upraven
In particular, \Hyper\ generates the factor rectangles by first computing for a given support $\alpha$ supplied by a user the set 
${F}_\alpha$ of all
$\alpha$-frequent itemsets from $I$. Then, \Hyper\ generates the factor rectangles from the set $\mathcal{C}_\alpha$  of 
rectangles corresponding to the itemsets in the set ${F}_\alpha$ enriched by all the 
singleton itemsets. 
For each rectangle in $\mathcal{C}_\alpha$, \Hyper\ finds the minimum cost rectangle by adding rows from a list of rows sorted by 
their coverage of the yet uncovered data. The rectangle with minimum cost over all the itemsets is then added
to the set $\mathcal{F}$ of factors.
Even though, as emphasized by the authors, \Hyper\ runs in time polynomial w.r.t. $\mathcal{C}_\alpha$,
its time is not polynomial w.r.t. to the input size because the size of $\mathcal{C}_\alpha$ may be exponential w.r.t.
to the input size---an important fact not mentioned by the authors.

%\paragraph{\GreConD}
\emph{\GreConD} \citep[Algorithm 2]{BeVy:Dof} %(denoted Algorithm 2 in that paper)  
performs a particular greedy search  ``on demand'' for formal
concepts of the input matrix $I$ and utilizes these concepts
to produce matrices $A_{\mathcal{F}}$ and $B_{\mathcal{F}}$.
This search improves the idea of the basic set cover algorithm in that
it avoids the necessity to compute all formal concepts of $I$, resulting
in orders of magnitude time saving while retaining the quality of decomposition.
Note also that \citep[Algorithm 1]{BeVy:Dof} implements the basic set cover algorithm,
thus proceeds essentially as the \Tiling\ algorithm but is considerably more time efficient
because the maximal rectangles are computed in advance.
\GreConD\ is a from-below decomposition algorithm, designed to compute
exact decompositions. When stopped after computing the first $k$ factors
or after the error $E$ does not exceed $\varepsilon$, \GreConD\ provides approximate solutions
to both DBP and AFP (Section \ref{sec:p}).
 
\emph{\PaNDa} (Patterns in Noisy Datasets) \cite{LuOrPe:Mtpbdn} is designed to 
solve a modification of DBP which consists in employing the MDL %M: preklep principle
principle. In particular,
for a given $I$ and $k$, the problem is to extract from $I$ a set $\mathcal{F}$ of $k$ patterns 
(pairs $\tu{C,D}$ of sets of rows and columns) minimizing the cost of $\mathcal{F}$.
The cost of $\mathcal{F}$ is the sum of description complexity of $\mathcal{F}$, defined
the same way as for \Hyper\ (see above), and the error $E(I,A_\mathcal{F}\circ B_\mathcal{F})$.
Each pattern $\tu{C,D}$ in $\mathcal{F}$ is computed by first computing its core and then 
extending the core to $\tu{C,D}$. A core is, in fact, a rectangle contained in $I$ and is 
computed by adding columns from a sorted list (sorting by several criteria is proposed).
Extension to $\tu{C,D}$ is performed by adding further columns and rows to a core
until such %MM: preklep additon -> addition
addition does not help in minimizing the cost.
The authors also propose simple randomization to overcome the drawback of 
selecting the columns from a fixed, sorted list.

%------------------------------------------------------------------
%\subsection{Datasets}

%\paragraph{Synthetic data}
\paragraph{Synthetic data}
We use the datasets described in Table \ref{tab:sets}.
Every Set $i$ consists of 1000 datasets of the given size,
obtained as Boolean products of matrices $A$ and $B$ that are generated
randomly with prescribed densities dens~$A$ and dens~$B$ and the corresponding dimensions 
(for instance, in Set 1, $A$ and $B$ are $300\times 20$ and $20\times 100$ matrices).
The average density of the matrices in Set 1, \dots, Set 6 was 
0.2, %set 1  20\%, 
0.05, %set2    5\%    nejřidší dataset
0.1, %set3    10\%
0.2, %set4    20\%
0.5, %set5    50\%    nejhustší dataset
0.3, %set6    30\%
respectively.
The data does not include noise, which
 is considered in Section \ref{sec:noise}.
%A given noise was then added to the resulting matrix $A\circ B$.
%That is, we use a similar scenario as in \cite{MiMiGiDaMa:TDBP}.

%Nakonec jsem to udelal stejne jako to ma Miettinen, vysledny dataset je vytvoren pronasobenim dvou matic A a B. Basis oznacuje vnitrni dimenzi matic dens A, respektive dens B oznacuje hustotu jednicek v matici A a B (rovnomerna hustota, pres radky respektive sloupce). Temito hustotami se da ridit celkova hustota vysledne matice (omezene). Noise oznacuje nahodne pridany sum. Nizke hodnoty basis jsou proto, aby v dalsich testech, ktere se provadeji pro $k=5, 10, 15, 20, 25, 30$ byly rozumne vysledky.

%M_: 3.12. v teto tabulce jsou finalni parametry datasetu, hustsi matice jsou kolem 18%, ridsi kolem 10%, nejridsi kolem 5%, chtelo by to detailne poexperimentovat jak parametry ovlivnuji hustotu vysledne matice, popripade to ykusit nejak teoreticky. Velikost zakladny je v duchu Mietinoveho scenare experimentu.

\begin{table}[h]
\centering
\begin{tabular}{lccccc}
\hline
dataset & size  & $k$ & dens $A$ & dens $B$ \\%& noise\\
\hline

Set 1 & 300$\times$100 		& 20  & 0.10 & 0.10 \\%& $<5\%$\\
Set 2 & 500$\times$250 		& 20  & 0.05 & 0.05 \\%& $<5\%$\\ 
Set 3 & 500$\times$250 		& 20  & 0.10 & 0.05 \\%& $<5\%$\\ 
Set 4 & 500$\times$250 		& 30  & 0.12 & 0.12 \\%& $<5\%$\\ 
Set 5 & 1000$\times$500 	& 50  & 0.10 & 0.10 \\%& $<5\%$\\ 
Set 6 & 10000$\times$1000 	& 50  & 0.10 & 0.10 \\%& $<5\%$\\ 

\hline
\end{tabular}
\caption{Synthetic datasets $I = A \circ B$}
\label{tab:sets}
\end{table}

%Parametry pro experimenty se šumem byly: hustota matice A 0.1, hustota matice B 0.1 (typově stejné jako Set1), 
%průměrná hustota kolem 20\%, zkoušel jsem samozřejmě i jiné, ale výsledky jsou srovnatelné.

%průměrné hustoty náhodných datasetů:
%set1    20\%
%set2    5\%    nejřidší dataset
%set3    10\%
%set4    20\%
%set5    50\%    nejhustší dataset
%set6    30\%

%napadá mě, že Asso se chová špatně na méně hustých a na více hustých datasetech a "rozumněji" na středně hustých datasetech, když tak procházím další výsledky experimentů s jinými konfiguracemi, tak to podporuje tohle tvrzení.
%
%RADIM: ZKUS TEDY JESTE PRO MATICE JINYCH ROZMERU  s hustotou kolem 10\% a 50\%, at vidime, jestli to je tim,
%asi nejlepe matice s rozmery a basis hodnotou jako pro Set 1, 3, 4, 6, at to muzeme srovnat.

%oproti Mietienovi mame vetsi hodnoty basis. Jeho jsou prilis male a jako takove prilis umele. Je jasne, ze tyhle hodnoty zvolil predevsim kvuli rychlosti, ty masivni datasety se jinak pocitaji velice dlouho.
%Co se tyka hustot matice, je to velice ruznorode. Pomerne spatne se predikuje jak bude matice husta v zavislosti na 4 parametrech. ALe obecne plati ze nizke hodnoty dens a basis znamenaji nizsi hustou.

%------------------------------------------------------------------
%\subsection{Real data}

%\paragraph{Real data}
%R4: data.
\emph{Real data.}
%M_: pozadavky od Miettinena: * DNA amplification: attached. You have to cite MYLLYKANGAS, S., HIMBERG, J., BĂ–HLING, T., NAGY, B., HOLLMĂ‰N, J., AND KNUUTILA, S. 2006. DNA copy number amplification profiling of human neoplasms. Oncogene 25, 55, 7324â€“7332.
% * Paleo: attached. You must cite FORTELIUS, M. ET AL. 2003. Neogene of the old world database of fossil mammals (NOW). http://www.helsinki.fi/science/now/.   This data is based on NOW public release 030717 M. Fortelius, Neogene of the Old World Database of Fossil Mammals (NOW â€™05), http://www.helsinki.fi/science/now/, 2005.
% takhle to ma Miettinen ve svem clanku a vzdy prislusne citace \footnote{NOW public release 030717, available from http://www.helsinki.fi/science/now/.}
We used the datasets 
Mushroom \cite{BaLi:UCI}, Chess \cite{BaLi:UCI}, 
%M: url adresy vysazeny \tt
DBLP\footnote{{\tt http://www.informatik.uni-trier.de/\raisebox{-2.5pt}{${}^\sim$}ley/db/}},
Paleo\footnote{NOW public release 030717, available from {\tt http://www.helsinki.fi/science/now/}.},  
Tic-tac-toe \cite{BaLi:UCI},   and 
DNA \cite{Myea:DNA}
with characteristics shown
in Table \ref{tab:sprd}, which are
 well known and  used in the literature on BMF.

%Asi bych zaradil Tic-tac-toe az jako posledni dataset (jeste se pocita DNA dataset, zitra - stred by mel byt hotov, takze pokud bude misto muzeme mit 6 realnych datasetu). Mozna bychom mohli zminit ze vsechny prvky jsou essential a ze i presto se algoritmus nechova nejak hazardne, ale stale pocita rozumnou faktorizaci.

%------------------------------------------------------------------
%\subsection{Observed characteristics}

%\paragraph{Reduction in size: $I$ vs. $\ess(I)$}
\subsection{Reduction in Number of Entries with $1$: $I$ vs. $\ess(I)$}
In view of the properties of essential parts and the strategy of %M: GreEss nebylo sazeno makrem
\GreEss, 
it is clearly significant to observe the ratio $||\ess(I)||/||I||$
of the number of entries containing $1$ in $\ess(I)$ to the corresponding number
for $I$. These characteristics are provided in
Table \ref{tab:spe} (synthetic data) and Table \ref{tab:sprd} (real data).
As one can see, the reduction in the number of $1$s in the essential parts is significant. 
%the characteristics (their average values in case of %M_: preklep, synthetic
%synthetic datasets) regarding
%essential parts, namely the numbers $||I||$ and $||\ess(I)||$ of $1$s in the original
%data and its essential part, and the ratio $||\ess(I)||/||I||$. These are clearly is important
%numbers in view of our results and the algorithm.

%Statisticke vlastnosti nahodnych datasetu. Prvni sloupec oznacuje prumerny pocet jednicek. 
%Druhy prumerny pocet ess policek. 
\begin{table}%[h]
\centering
\begin{tabular}{lccc}
\hline
dataset & avg $||I||$ & avg $||\ess(I)||$ & avg $||\ess(I)||/||I||$ %& std \\%& var \\
\\
\hline

Set 1 & 5039   & 2764  & 0.549 \\%& 0.025 \\%& 0.006\\
Set 2 & 11966  & 6412  & 0.536 \\%& 0.020 \\%& 0.004\\
Set 3 & 22841  & 5207  & 0.228 \\%& 0.013 \\%& 0.001\\
Set 4 & 44027  & 8894  & 0.202 \\%& 0.023 \\%& 0.005\\
Set 5 & 195990 & 34005 & 0.174 \\%& 0.014 \\%& 0.002\\
Set 6 & 3907433 & 47359 & 0.012 \\%& 0.024  \\%& 0.006 \\

\hline
\end{tabular}
\caption{Essential part of $I$ (synthetic data)}
\label{tab:spe}
\end{table}

%M: v teto tabulce je zakomentovany tic_tac_toe, ale dale jej pouzivame
%R4: zviditelneno
\begin{table}%[h]
\centering
\begin{tabular}{lcccc}
\hline
dataset & size & $||I||$ & $||\ess(I)||$ & $||\ess(I)||/||I||$ \\%& dupr & dupc & redr & reds \\
\hline

Mushroom & 8124$\times$119 & 186852 & 82965 & 0.444\\% & 0 & 6 & 0 & 13\\
DBLP & 19$\times$6980& 40637 & 1601 & 0.039 \\%0 & 6090 & 0 & 626\\
Paleo & 501$\times$139 & 3537 & 1906 & 0.539 \\%30 & 0 & 79 & 0\\
Chess & 3196$\times$76 & 118252 & 71296 & 0.603 \\%0 & 0 & 0 & 4\\
DNA & 4590$\times$392 & 26527 & 1685 & 0.064 \\%& 21 & 617 & 8 \\
Tic-tac-toe & 958$\times$29 & 9580 & 9580 & 1 \\%0 & 0 & 0& 1 \\

\hline
\end{tabular}
\caption{Essential part of $I$ (real data)}
\label{tab:sprd}
\end{table}

%MARTIN:
%3) pomery $B_ess(I)/B(I)$
%tady je tabulka s výsledky, už si vzpomínám proč jsme ji nezařazovali, protože ty výsledky jsou poměrně nepříznivé, ale počítám teď výsledky, které ukazují distribuci počtu pokrytých ess políček mezi koncepty, to je mnohem zajímavější údaj, ale musím zpracovat výsledky a nakreslit histogramy. Tohle bude chtít větší komentář, takže bude asi lepší když tyto experimenty přinesu v úterý a okomentuju je slovně.
%
%
%\begin{table}[h!]
%\begin{center}
%\begin{tabular}{lccc}
%\hline
%Dataset & Počet ne Ess  & $||\mathcal{B}(I)||$ & $||\mathcal{B}_{\ess}(I) /
%\mathcal{B}(I)||$ \\
%\hline
%Mushroom & 29 & 238710 & 0.999\\
%DBLP & 2476 & 2495 & 0.007\\
%Paleo & 344 & 10225 & 0.966\\
%DNA & 3331 & 4483 & 0.257\\
%Tic Tac Toe & 2 &59505 & 0.999\\
%\hline
%%pocet ne Ess konceptu / pocet konceptu
%%Mushroom & 29/238710\\
%%DBLP & 2476/2495\\
%%Paleo & 344/10225\\
%%DNA & 3331/4483\\
%%Tic Tac Toe & 2/59505\\
%\end{tabular}
%\end{center}
%\end{table}
%
%počtem ne Ess je myšleno počet konceptů, které nekryjí žádné Ess políčko.

\subsection{Performance of Algorithms}
\label{sec:pa}

Arguably, the most important aspect in evaluating the performance of BMF algorithms
is the quality of decompositions delivered by the algorithms. 
The existing literature provides numerous
evidence demonstrating the the factors in Boolean data are meaningful and interesting
from data analysis point of view, hence we focus on comparison in terms of quantitative
criteria described below.
Recall that \GreEss\ is designed for the AFP problem.
However, we take into account both views reflecting the goals of DBP and AFP, see Section \ref{sec:nbn},
and require that a good factorization algorithm computes a decomposition (or approximate decomposition)
of the input matrix $I$ using a reasonably small number of factors in such a way that the 
first factors have a reasonably good coverage, i.e. explain a large portion of data. 
For this purpose we compare the factorization algorithms using their coverage quality introduced below.
In addition to the quality of decompositions, another issue is the time complexity of the  algorithms.
We consider this issue in the next paragraph.

\paragraph{Time complexity}
Time complexity is not as critical a constraint as the quality of delivered decompositions, though clearly,
time complexity should not be prohibitive.
Since %M: preklep time
time complexity is not our primary concern, we postpone its detailed analysis, including 
the analysis of ``bad cases'' for the particular algorithms,
to future work and provide just basic observations.
We implemented all algorithms in MATLAB with critical parts written C a compiled to binary MEX files.
We employed about the same level of optimization to make the time demands of the algorithms
comparable. 
For information about the time complexity of \Tiling, \Asso, \Hyper, \GreConD, and \PaNDa\
we refer the above-mentioned papers.
It is easy to see that the time complexity of \GreEss\  is polynomial in terms of the number of 
rows and columns of the input matrix $I$ and, in fact, is of the same order as that of \GreConD.
This is because in \GreEss,  both \ComputeIntervals\ and the subsequent computing of 
$\mathcal{F}$ from $\mathcal{G}$ has asymptotically the same time complexity as \GreConD\
for the following reasons.
In \ComputeIntervals, computing  $\ess(I)$ is simple and the subsequent computing of $\mathcal{G}$ proceeds
by extending attribute concepts in $\ess(I)$ (lines 2--11), which has the complexity of the same order as the one
of \GreConD. In addition, the greedy search in computing $\mathcal{F}$ (lines 3--22 in \GreEss)
proceeds by extending similarly the attribute concepts in the context $J$. Such extension 
is, in the worst case, of the same order as the time %M: preklep complexity
complexity of  \GreEss, again.

\Tiling, \GreConD, and \GreEss\ run without parameters to be set. For the other algorithms,
we followed the recommendations by the authors. We, however, experimented
with setting the parameters and chose them individually, with
%R4: the
the best performance for 
%R4: a -> every
every given dataset.
%The dependence on parameters may be considered both advantage 
%(adaptivity) and disadvantage (burden on user's side).
In particular, \Asso\ requires us to set $\tau$, and (one of) $w^+$ and $w^-$ (see above).
In most cases, the best choice was $0.8\leq \tau < 1$ and $w^+,w^-\in\{1,2,3\}$.
%The dependence of \Asso on parameters may be considered both advantage 
%(adaptivity) and disadvantage (burden on user's side). We obtained theoretical results
%on \Asso's ability to compute exact decompositions but omit them due to limited scope.
For \Hyper, we set the support parameter $\alpha\geq 0.3$ and used closed $\alpha$-frequent
itemsets (see above). For \PaNDa, we used attribute sorting by frequency and the randomization
described in \cite{LuOrPe:Mtpbdn}.
These settings are used in the evaluation below.

The overall fastest algorithm is \GreConD. This algorithm does not perform any data preprocessing
and utilizes a very fast heuristic for computing the factors.
Second to \GreConD\ is \GreEss\ which was about $2\times$ slower.
%The reason is that \GreEss\ first computes the intervals
Third to \GreConD\ is \Asso\ which was about $3$--$4\times$ slower than \GreConD.
Fourth and fifth in terms of time demand are \Hyper\ and \PaNDa\ which are about $5\times$
slower than \GreConD.
However, the time consumed by \Hyper\ depends on the size of the set $\mathcal{C}_\alpha$ of frequent 
itemsets, and hence depends on $\alpha$ (see above). As is well-known, the number of frequent
as well as closed frequent itemsets may be exponential in the number of items. 
As a result, the worst case time complexity of \Hyper\ is exponential in the number of attributes,
as mentioned above.
\Tiling\ is the slowest of all the compared algorithms.  On average, it was about $400\times$ slower than
\GreConD.
This is because in selecting each tile, \Tiling\ browses the set of all maximal tiles which is
usually very large and may be exponential in terms of the minimum of the number of objects
and attributes.
Note also  that according to \citep{BeVy:Dof} and our experience, \GreConD\ implemented in 
C factorizes Mushroom dataset in the order of seconds on an ordinary PC.

\paragraph{Quality of decompositions}
To assess the quality of decompositions, we employed the following function
of $A\in\{0,1\}^{n\times l}$ and $B\in\{0,1\}^{l\times m}$ 
representing the \emph{coverage quality} of the first $l$ factors delivered by the particular algorithm:
\begin{equation}\label{eqn:cq}
%  c=1-\frac{E(I,A\circ B)}{||I||}.
  c=1-{E(I,A\circ B)}/{||I||}.
\end{equation} 
Similar functions are used in \citep{BeVy:Dof,GeGoMi:Td}.
%RRR: pridano
We observe the values of $c$ for $l=0,\dots,k$, where $k$ is the number of factors
delivered by a particular algorithm.
Clearly, for $l=0$ (no factors added, $A$ and $B$ are ``empty'') we have $c=0$. 
It is desirable that for $l=k$ we have $I=A\circ B$, i.e. the data is fully explained
by all the $k$ factors computed, in which case $c=1$. 
For a good factorization algorithm,
%In accordance with the beginning of Section \ref{sec:pa}, 
$c$ should be increasing in $l$  and should have relatively large values even for small $l$, corresponding
to the requirements that as we add factors, the error decreases, and that
the first factors explain a large portion of data, respectively.

%RRR: zapoznamkovano
%For from-below decomposition algorithms, such as
%\GreEss, \GreConD, and \Tiling, we clearly have 
%%$c=1-\frac{E_u(I,A\circ B)}{||I||}$.
%$c=1-{E_u(I,A\circ B)}/{||I||}$.
%For \Asso, which is not a from-below decomposition algorithm,
%we observe $c$ but also %$c_u=1-\frac{E_u(I,A\circ B)}{||I||}$  
%$c_u=1-{E_u(I,A\circ B)}/{||I||}$  
%and
%%$e_o=\frac{E_o(I,A\circ B)}{||I||}$
%$e_o={E_o(I,A\circ B)}/{||I||}$
% (column \Asso: $c_u-e_o$) . Clearly,
%$c=c_u-e_o$. This provides us with information about both types
%of error %M_: preklep, committed
%committed and makes it easy to compare the parts due to ``uncovered'' $1$s
%for \Asso ($c_u$) with those of \GreEss, \GreConD, and \Tiling ($c$).

The results for synthetic and real data are shown in Fig.~\ref{fig:cs}, Table~\ref{tab:qs}, 
and
Table~\ref{tab:qr}.
%, and Fig.~\ref{fig:cr}.  
The results for synthetic data are obtained
as averages over the 1000 datasets comprised by each Set $i$.
In Table~\ref{tab:qs}, the performance of the algorithms is represented
by the coverage quality  of the sets of the first $k$ factors computed by the algorithm for selected $k$.
In every row, the best performance is shown in bold.
In Fig.~\ref{fig:cs}, we display the curves of the coverage quality as a function of 
$k$.
We do not display results for other data and other parameters
of synthetic data; the presented results are, however, representative w.r.t. assessment of
quality of decompositions of the six algorithms compared.
In Table ~\ref{tab:qr}, we display the number of factors needed to cover 
$25\%$,  $50\%$, $75\%$, and $100\%$ of the data for the six real
%R4: dataset -> datasets
 datasets.

%M2: mezery mezi obrazky jsem musel zvetsit byly prekrite popisky o 2pt
\begin{figure}%[h]%[htp]
\centering
\subfloat[Set 1]
{\label{fig:aproxset1}\includegraphics[scale=0.45]{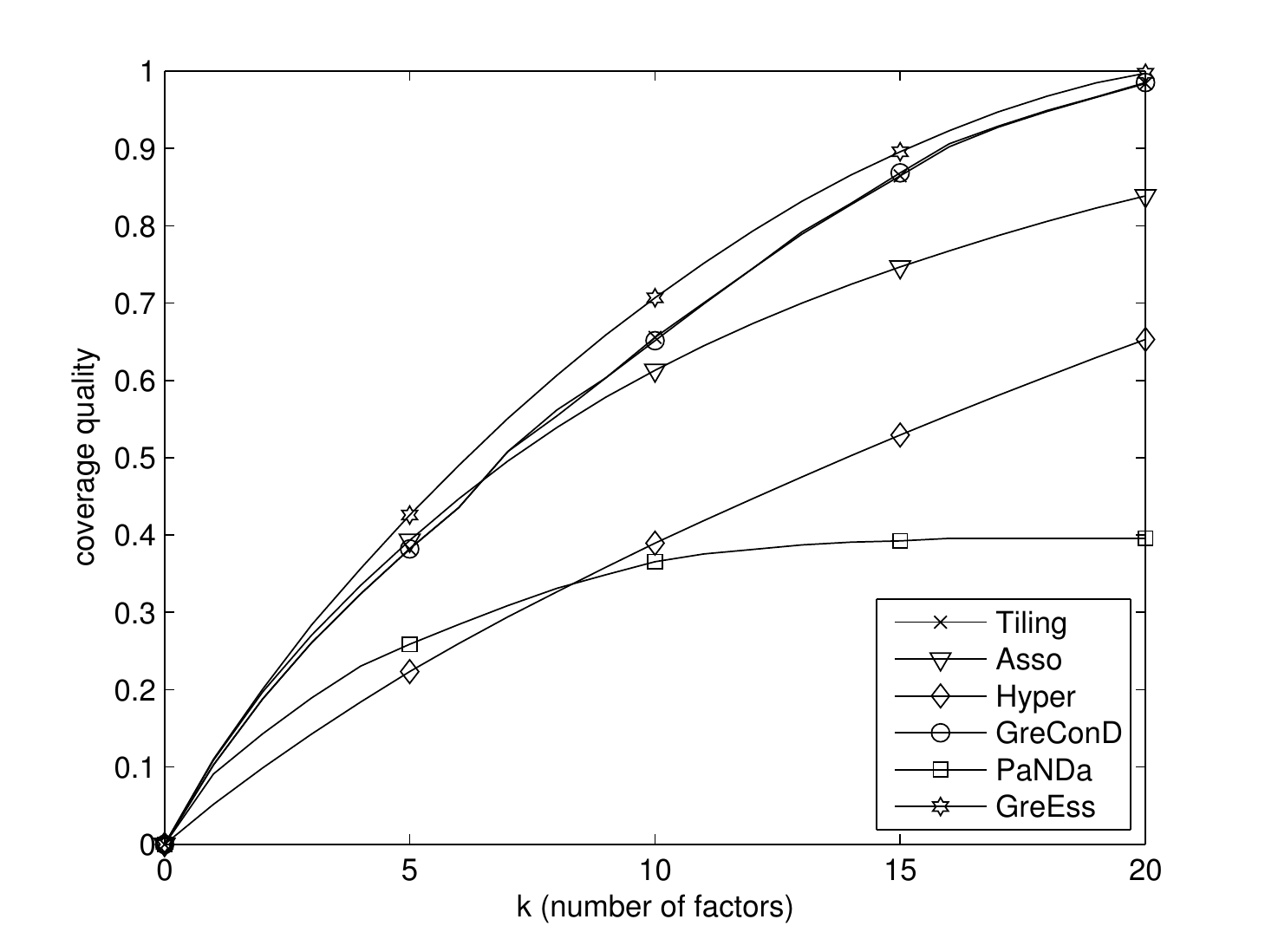}}
\hspace{0pt}
\subfloat[Set 2]
{\label{fig:aproxset2}\includegraphics[scale=0.45]{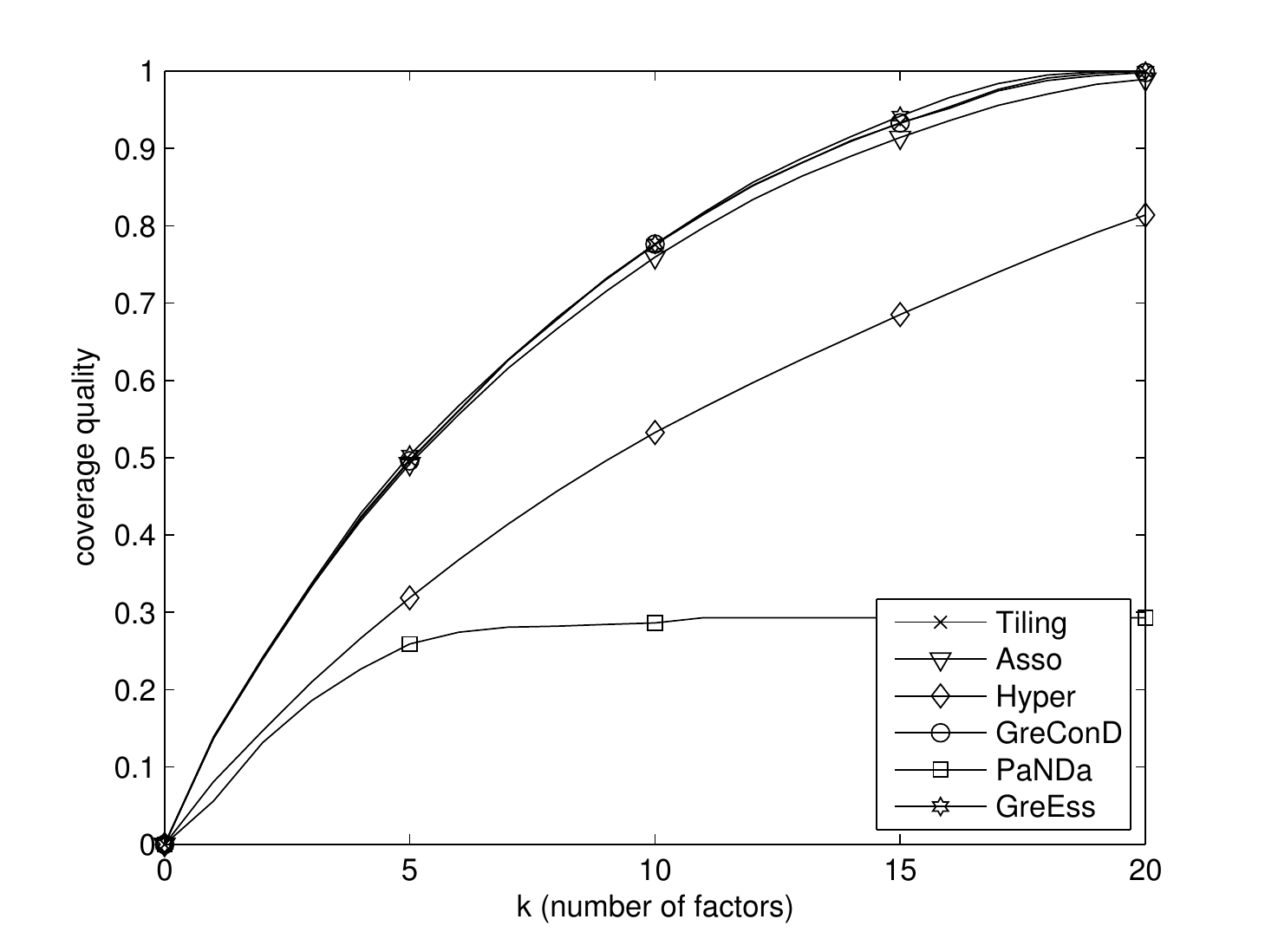}}
\vspace{-2pt}
\subfloat[Set 3]
{\label{fig:aproxset3}\includegraphics[scale=0.45]{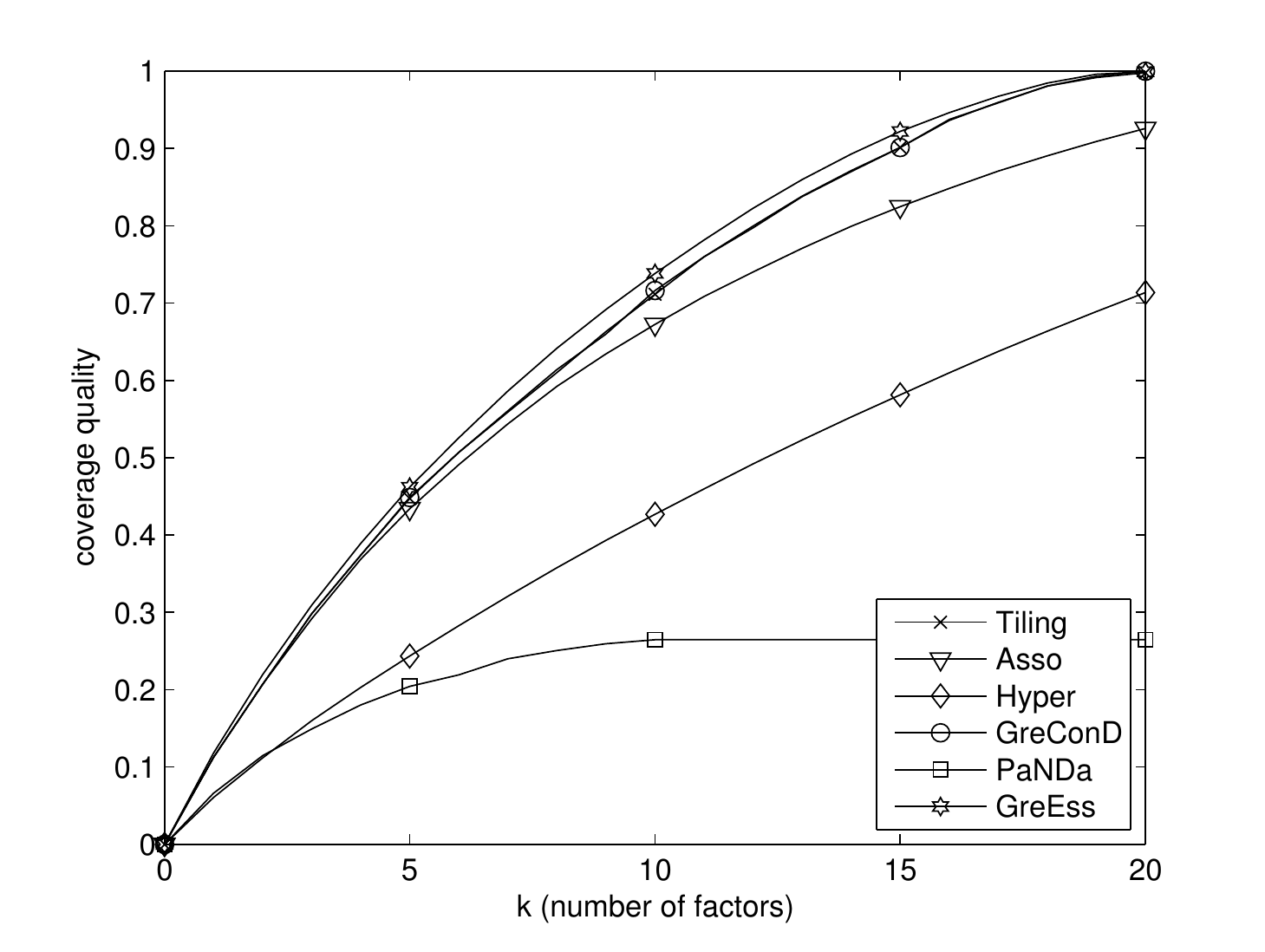}}
\hspace{0pt}
\subfloat[Set 4]
{\label{fig:aproxset4}\includegraphics[scale=0.45]{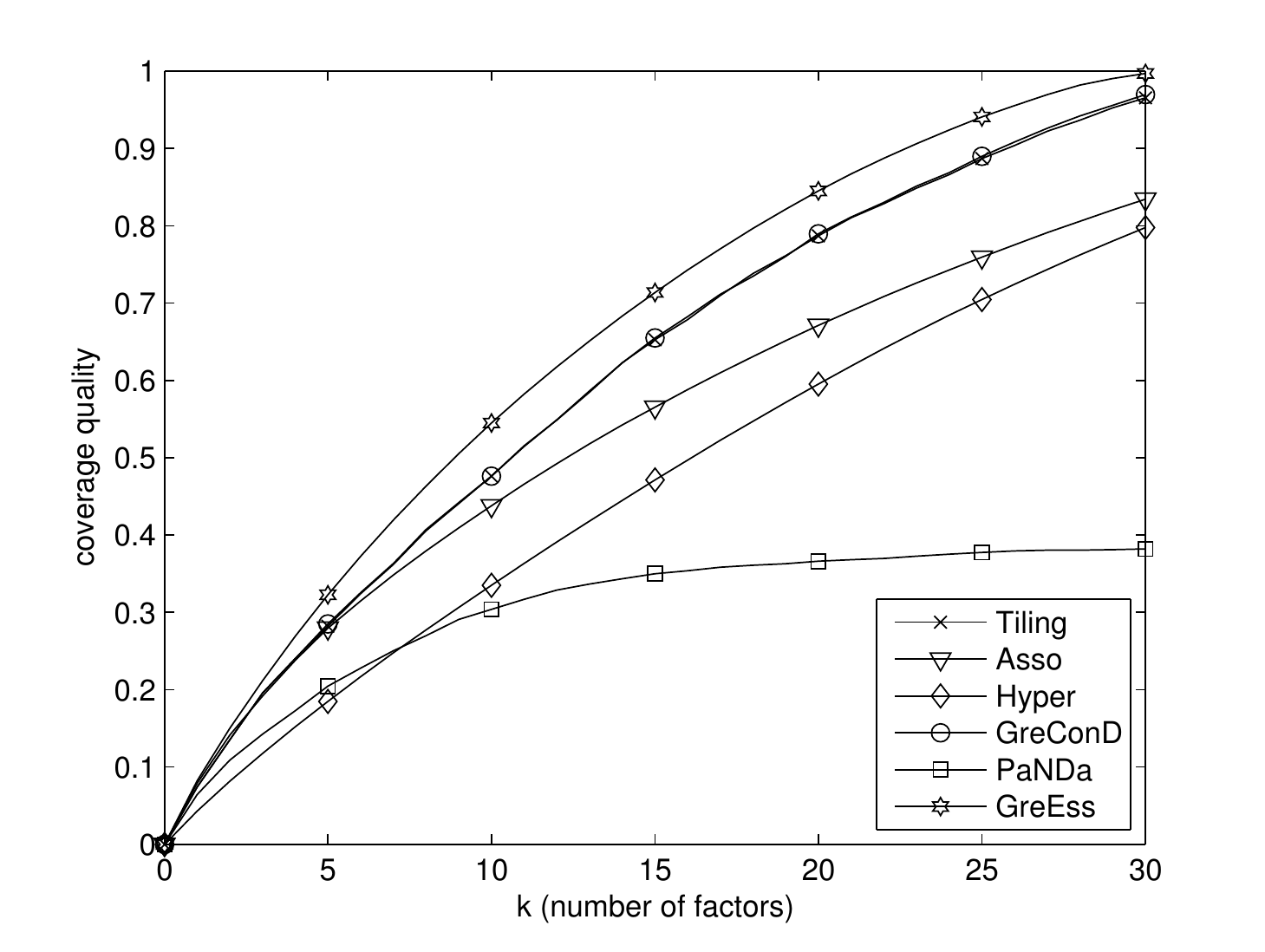}}
\vspace{-2pt}
\subfloat[Set 5]
{\label{fig:aproxset5}\includegraphics[scale=0.45]{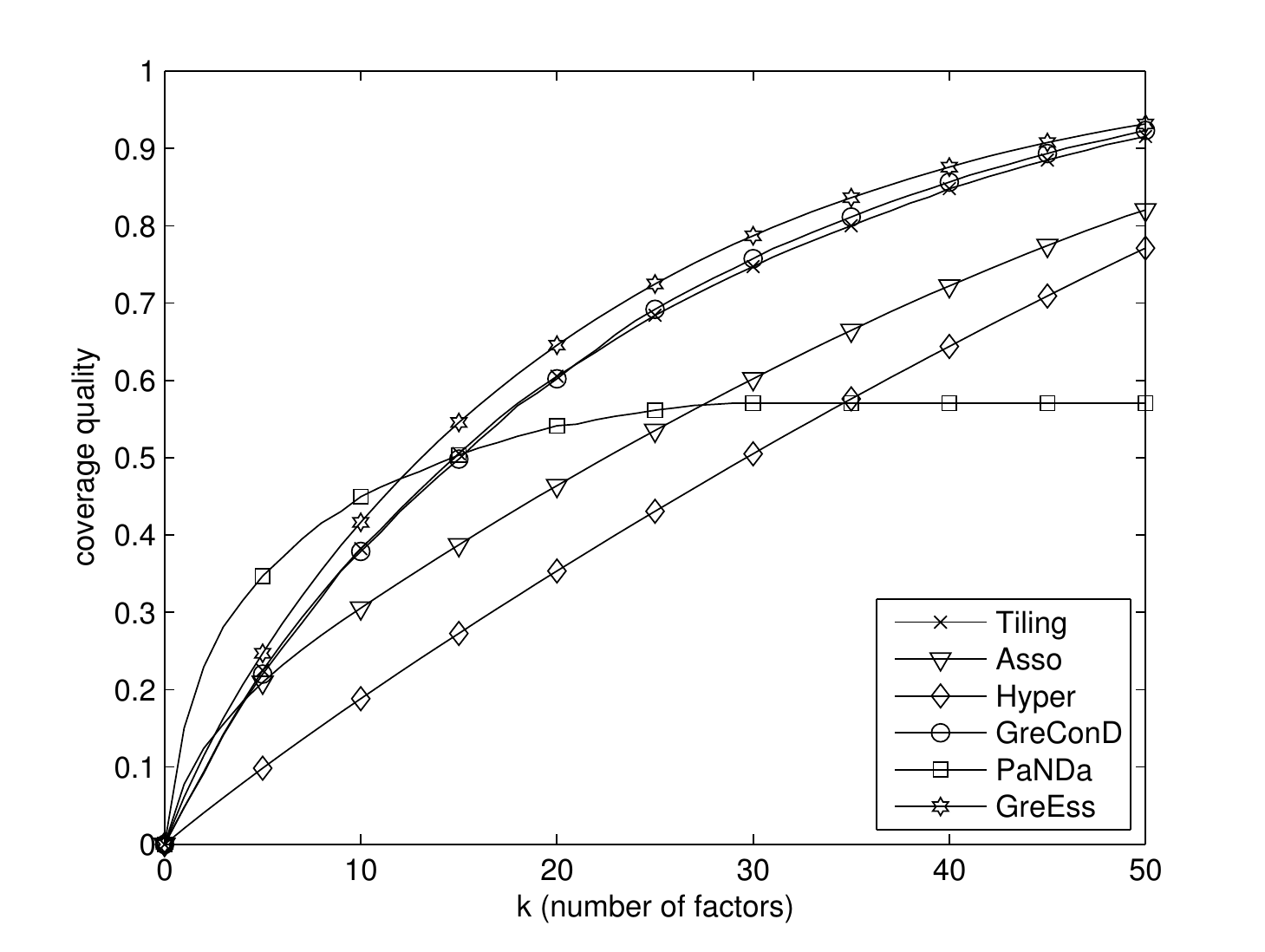}}
\hspace{0pt}
\subfloat[Set 6]
{\label{fig:aproxset6}\includegraphics[scale=0.45]{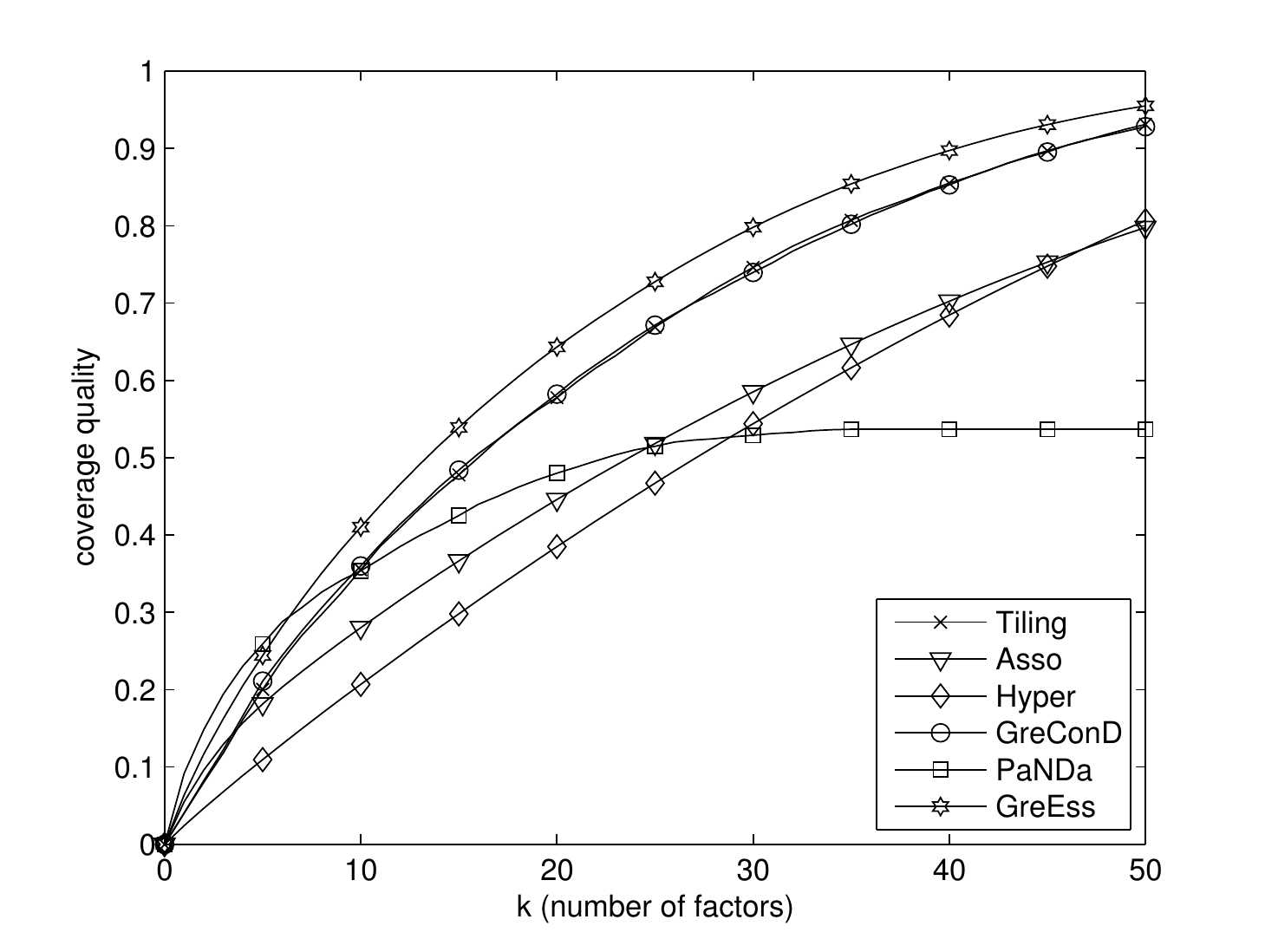}}
\caption{Coverage quality of the first $k$ factors (synthetic data).}
\label{fig:cs}
\end{figure}

%The graphs let us see the  global picture but do not show well the numbers of factors
%needed for exacts decomposition, i.e. the approximations of $\mathrm{rank}_\mathrm{B}(I)$
%%of the input data 
%computed by the algorithms. These numbers are better seen from the tables.

%[Poznamka k parametrum \Asso: Aby \Asso nebylo prilis znevyhodneno, bylo u realnych datasetu zkouseno vice variant nastaveni $\tau = 0.95, 0.9, 0.85$ a nastaveni vah $w^+ =1, 2, 3 a w^- 1, 2, 3$. Neni to sice uplne nejkorektnejsi, spravne by se tyto hodnoty mely hledat pulenim intervalu, ale v zasade tohle jsou nejrozumnejsi hodnoty pri kterych to dava dobre vysledky (prakticke zkusenosti).
%hodnota $k$ pak byla nastavena na vetsi hodnotu nez je pocet atributu daneho kontextu (jen kvuli grafum).]
%
%For the values of the parameters recommended in \citep{MiMiGiDaMa:TDBP},
%\Asso as a rule is not capable of computing an exact factorization of 
%the input matrix. For the lack of space, we only mention here that 
%one can prove that \Asso is capable of finding exact decompositions
%only for extreme setting of its parameters (e.g. for $\tau=1$ and 
%$w^->w^+(m-1)$) in which case, however, the quality of decompositions
%delivered is poor.

\begin{table}%[h]
\centering
\begin{tabular}{ccccccccc}
\hline
%M2: zarovnal jsem na stred popis atributu
& & \multicolumn{6}{l}{coverage $c$ of the first $k$ factors}\\
dataset & $k$ & \Tiling & \Asso & \Hyper & \GreConD & \PaNDa & \GreEss \\
\hline
Set 1 
& 5    &0.3820    &0.3929    &0.2234    &0.3820    &0.2586    &\bf 0.4260 \\
&10    &0.6557    &0.6131    &0.3894    &0.6512    &0.3654    &\bf 0.7068 \\
&15    &0.8645    &0.7467    &0.5294    &0.8686    &0.3925    &\bf 0.8957 \\
&20    &0.9839    &0.8387    &0.6529    &0.9852    &0.3958    &\bf 0.9971 \\
&25    &\bf 1.0000    &0.9049    &0.7602    &\bf 1.0000    &0.3958    &\bf 1.0000 \\
\hline

Set 2 
& 5    &0.4970    &0.4919    &0.3187    &0.4961    &0.2591    &\bf 0.5035 \\
&10    &0.7755    &0.7594    &0.5326    &\bf 0.7764    &0.2864    &\bf 0.7764 \\
&15    &0.9326    &0.9141    &0.6850    &0.9330    &0.2933    &\bf 0.9418 \\
&20    &\bf 1.0000    &0.9894    &0.8139    &0.9982    &0.2933    &\bf 1.0000 \\
&25    &\bf 1.0000    &0.9977    &0.9090    &\bf 1.0000    &0.2933    &\bf 1.0000 \\
\hline

Set 3 
& 5    &0.4471    &0.4341    &0.2435    &0.4485    &0.2041    &\bf 0.4620 \\
&10    &0.7114    &0.6725    &0.4268    &0.7163    &0.2647    &\bf 0.7384 \\
&15    &0.9011    &0.8246    &0.5814    &0.9011    &0.2647    &\bf 0.9219 \\
&20    &0.9980    &0.9259    &0.7136    &\bf 1.0000    &0.2647    &0.9991 \\
&25    &\bf 1.0000    &0.9812    &0.8176    &\bf 1.0000    &0.2647    &\bf 1.0000 \\
\hline    

Set 4 
& 5    &0.2822    &0.2794    &0.1849    &0.2851    &0.2047    &\bf 0.3228 \\
&10    &0.4760    &0.4379    &0.3352    &0.4761    &0.3039    &\bf 0.5450 \\
&20    &0.7869    &0.6711    &0.5953    &0.7894    &0.3661    &\bf 0.8450 \\
&30    &0.9655    &0.8344    &0.7978    &0.9698    &0.3820    &\bf 0.9969 \\
&40    &\bf 1.0000    &0.9381    &0.9414    &\bf 1.0000    &0.3824    &\bf 1.0000 \\
\hline

Set 5 
& 5    &0.2251    &0.2095    &0.0984    &0.2203    &\bf 0.3471    &0.2471 \\
&15    &0.5053    &0.3873    &0.2729    &0.4983    &0.5034    &\bf 0.5455 \\
&30    &0.7471    &0.6021    &0.5050    &0.7575    &0.5706    &\bf 0.7871 \\
&50    &0.9154    &0.8206    &0.7712    &0.9234    &0.5706    &\bf 0.9319 \\
&60    &0.9577    &0.8867    &0.8830    &0.9652    &0.5706    &\bf 0.9666 \\
\hline

Set 6 
& 5    &0.2004    &0.1817    &0.1096    &0.2110    &\bf 0.2592    &0.2443 \\
&15    &0.4779    &0.3666    &0.2980    &0.4841    &0.4253    &\bf 0.5391 \\
&30    &0.7462    &0.5851    &0.5439    &0.7396    &0.5290    &\bf 0.7980 \\
&50    &0.9310    &0.7978    &0.8069    &0.9281    &0.5368    &\bf 0.9552 \\
&60    &0.9752    &0.8716    &0.9079    &0.9734    &0.5368    &\bf 1.0000 \\

\hline
\end{tabular}
\caption{Quality of decompositions (synthetic data).%Coverage quality of the first $k$ factors (synthetic data).
}
\label{tab:qs}
\end{table}

\begin{table}%[h]
\small
\centering
\begin{tabular}{cccccccc}
\hline
& coverage  & \multicolumn{6}{c}{number of factors needed for the prescribed coverage}\\
dataset & ($100c\%$) & \Tiling & \Asso & \Hyper & \GreConD & \PaNDa  & \GreEss \\
\hline

Mushroom  								
&	25\%     	&3	&2	&8	&3	&1	&2     \\	
&	50\%     	&7	&6	&19	&7	&NA        	&8     \\	
&	75\%     	&24	&36	&37	&24	&NA        	&26     \\	
&	100\%    	&119	&NA    	&122	&120	&NA        	&105    \\ 	
\hline								
								
DBLP 								
&	25\%     	&2	&2	&2	&2	&NA          	&2     \\	
&	50\%     	&5	&5	&5	&5	&NA         	&5     \\	
&	75\%     	&10	&10	&10	&11	&NA        	&10     \\	
&	100\%    	&21	&19	&19	&20	&NA        	&19    \\ 	
\hline								
								
Paleo 								
&	25\%     	&16	&16	&14	&16	&NA        	&15     \\	
&	50\%     	&39	&40	&38	&39	&NA          	&38     \\	
&	75\%     	&75	&76	&73	&76	&NA          	&73     \\	
&	100\%    	&151	&NA    	&139	&152	&NA        	&145    \\	
\hline								
								
Chess 								
&	25\%     	&2	&1	&9	&1	&1	&1     \\	
&	50\%     	&5	&2	&26	&4	&NA           	&6     \\	
&	75\%     	&16	&15	&39	&15	&NA         	&17     \\	
&	100\%    	&124	&NA    	&90	&124	&NA          	&113    \\	
\hline								
								
DNA 								
&	25\%     	&8	&6	&24	&8	&NA          	&13     \\	
&	50\%     	&32	&27	&67	&33	&NA         	&41     \\	
&	75\%     	&94	&80	&155	&96	&NA        	&105     \\	
&	100\%    	&489	&NA    	&392	&496	&NA      	&408    \\	
\hline								
								
Tic-tac-toe								
&	25\%     	&5	&6	&5	&5	&NA         	&5     \\	
&	50\%     	&12	&12	&11	&12	&NA        	&12     \\	
&	75\%     	&19	&19	&18	&19	&NA        	&19     \\	
&	100\%    	&31	&29	&29	&32	&NA        	&32    \\

\hline
\end{tabular}
\caption{Quality of decompositions (real data).}
\label{tab:qr}
\end{table}

All the algorithms  compute the factors for a given $I$
one after another. In case of \Tiling, \Hyper, \GreConD\ and \GreEss, this process is guaranteed 
to stop when an exact decomposition $I=A\circ B$ is found. 
With \Asso\ and \PaNDa, it often happens that an exact decomposition is not found and that the algorithm stops
with a relatively small coverage $c$ (i.e. large error $E(I,A\circ B)$), which is seen 
from the tables and graphs and is indicated by NA in Table~\ref{tab:qr}.
This is in particular true of \PaNDa.
This feature, which is a consequence of committing the error $E_o$ (cf. Observation \ref{thm:E}
and the discussion below), is a disadvantage of \Asso\ and \PaNDa\ when a large coverage is required.
On the other hand, \Asso\ tends to have a good coverage by the first couple of factors.
\Asso\ performs better on datasets which are sparse or dense compared to other datasets,
which can be observed on Set 2 and Set 5. 
\PaNDa\ tends to have a good coverage by the first couple of factors on dense datasets
which is seen in case of Set $5$.
\Hyper\ performs well with respect to the first quality criterion, namely the total number of factors
needed for an exact decomposition of the input matrix.
For the synthetic datasets, \Hyper\ is the fourth best, behind \Tiling, \GreConD, and \GreEss,
with \GreEss\ being the best one.
For the six real datasets, \Hyper\  is comparable to \GreEss\ in terms of the first quality criterion. 
%and actually provides the smallest number of factors.
However,
Tables~\ref{tab:qs} and \ref{tab:qr} and  the slowly-growing curves of coverage quality in Fig.~\ref{fig:cs}
reveal a significant drawback of \Hyper, namely a poor coverage by the
set of the first $k$ factors, even for a relatively large $k$.
The reason for this behavior is the following.
\Hyper\ includes in the set $\mathcal{C}_\alpha$ (see the above description of \Hyper)  not only
the rectangles corresponding to $\alpha$-frequent itemsets but also those corresponding to 
all the singleton itemsets. Including the singleton itemsets guarantees that an exact decomposition
of the input matrix is found when \Hyper\ computes them from $\mathcal{C}_\alpha$.
It turns out from the results, however, that the factors corresponding to the singleton items,
i.e. the rectangles induced by the columns of the input matrix are used very often. 
This causes a very low coverage by the sets of the first $k$ factors of \Hyper\ compared to the other algorithms.
Note in this connection that a trivial factorization algorithm that outputs for an input matrix 
$I\in\{0,1\}^{n\times m}$  the set $\mathcal{F}$ containing the $m$ rectangles corresponding to the columns
of $I$ will have a similar behavior in a sense, namely a slowly-growing curve of coverage quality which, nevertheless,
reaches full coverage (exact decomposition) with $k\leq m$.
None of \Tiling, \GreConD, and \GreEss\ suffers from this drawback of \Hyper.
\Tiling\ and \GreConD\ perform very similarly, confirming the evaluation results
of Algorithm 1 and \GreConD\ in \citep{BeVy:Dof} (cf. the description of \GreConD\ above).
One can see from the results that \GreEss\ performs best of these three algorithms on both synthetic and
%R4: uprava
real datasets, outperforming them  significantly, particularly in terms of the number of factors needed for exact decomposition.

%As we can see from the results, \GreEss performs best on  almost all the synthetic 
%and real data.
%In particular, \GreEss delivers the best approximation of $\mathrm{rank}_\mathrm{B}(I)$.
%%RR:
%%in most cases, being outperformed only by \Asso in case of Tic-tac-toe (a somewhat untypical
%%dataset according to the characteristics in Table \ref{tab:sprd}). 
%
%\Asso gives generally good results w.r.t. coverage by the first $k$ factors for relatively
%small $k$, for which the %M_: preklep, committed
%committed $E_o$ error is well %M_: preklep, compnesated
%compensated by a big drop
%in $E_u$. As mentioned above, however, \Asso generally does not perform well when
%large coverage, or exact decomposition---the limit case, is required.

From the point of view of the new strategy of \GreEss, which is based on the results regarding essential part of $I$, 
the following conclusions may be drawn.
%The geometry of BMF tells us that to decompose a Boolean matrix $I$ is equivalent to
%cover all $1$s by rectangles. 
Contrary to \Tiling, \Asso, \Hyper, \GreConD\ and \PaNDa, which all use different strategies of greedy coverage, 
but all aim at covering the most of the uncovered $1$s in $I$, \GreEss\ proceeds differently.
In its greedy coverage, \GreEss\ focuses on the essential $1$s in $I$ and considers them
as ``seeds'' of good factors. Such strategy is theoretically well justified, is fast, and leads
to improvement in quality of Boolean matrix factorization.

%--------------------
\subsection{Performance on Synthetic Data With Noise}
\label{sec:noise}

Noise in Boolean data is an issue discussed in the papers on BMF, see e.g.  
\cite{MiMiGiDaMa:TDBP,LuOrPe:Mtpbdn}.
In particular,  \PaNDa\ has been designed with the aim to perform well for data with noise.
The capability to factorize noisy data with \Asso\ has been demonstrated in \cite{MiMiGiDaMa:TDBP}.
In this section we provide the performance evaluation of \GreEss\ for noisy data and compare it with
\Asso\ and \PaNDa.
We use a scenario similar to those of \cite{MiMiGiDaMa:TDBP,LuOrPe:Mtpbdn}.
We performed the evaluation on synthetic datasets which are obtained by adding noise
to the datasets generated as those comprising Sets $1,\dots,6$ in Section \ref{sec:ac}.
In particular, we display the results for the datasets obtained by the same parameters  as those for 
Set $2$. The results are similar for the other parameters.
%Parametry pro experimenty se šumem byly: hustota matice A 0.1, hustota matice B 0.1 (typově stejné jako Set1), 
%průměrná hustota kolem 20\%, zkoušel jsem samozřejmě i jiné, ale výsledky jsou srovnatelné.

We observed the coverage quality of the datasets in a similar way as in Section \ref{sec:pa}.
The results are displayed in Figure \ref{fig:cr}.

%M: obrazek jsem posunul aby nezasahoval do conclusions
\begin{figure}[ht!]
\centering
\subfloat[\scriptsize Aditive noise---\GreEss]
{\label{fig:}\includegraphics[scale=0.31]{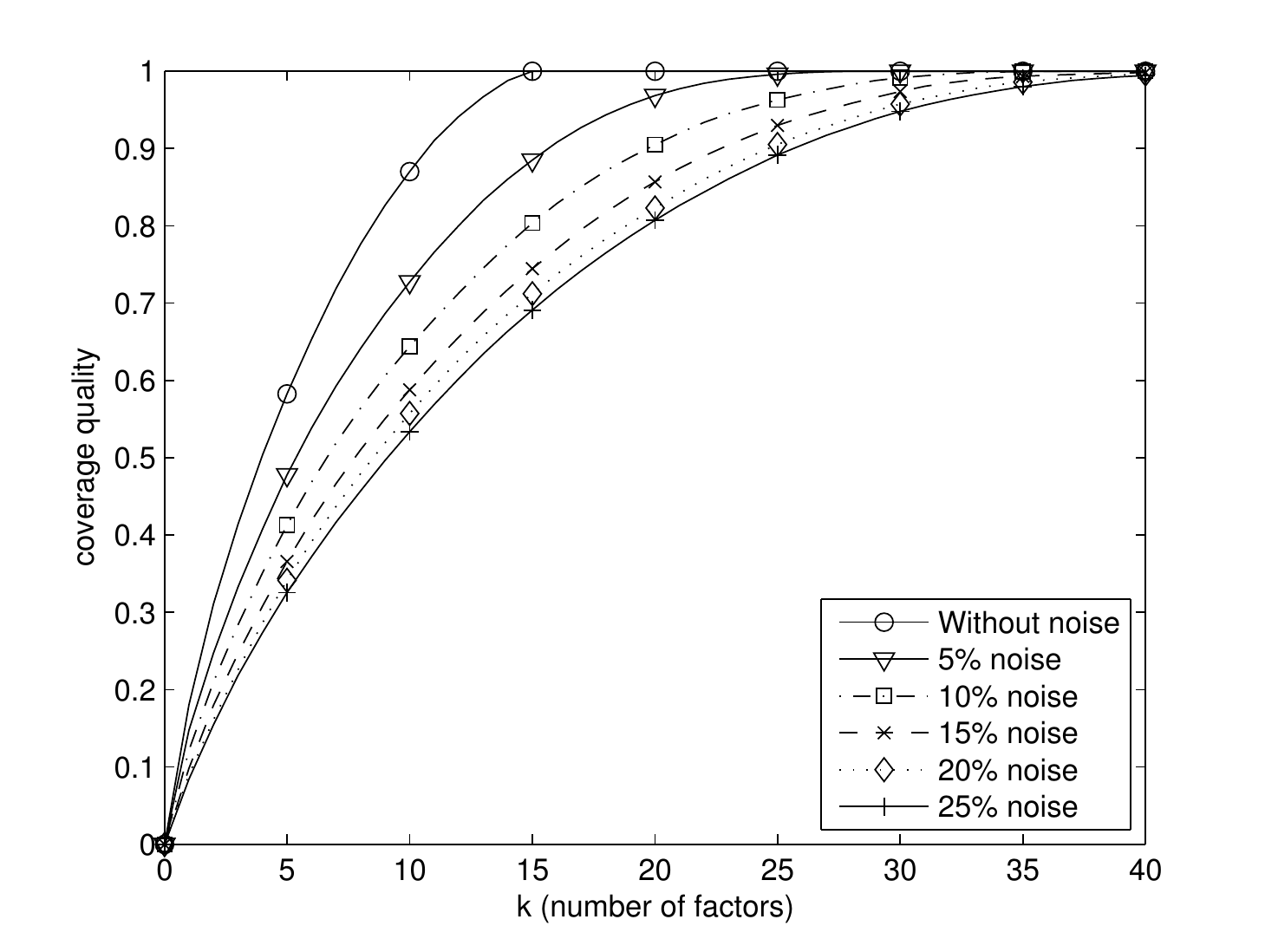}}
\subfloat[\scriptsize Aditive noise---\Asso]
{\label{fig:}\includegraphics[scale=0.31]{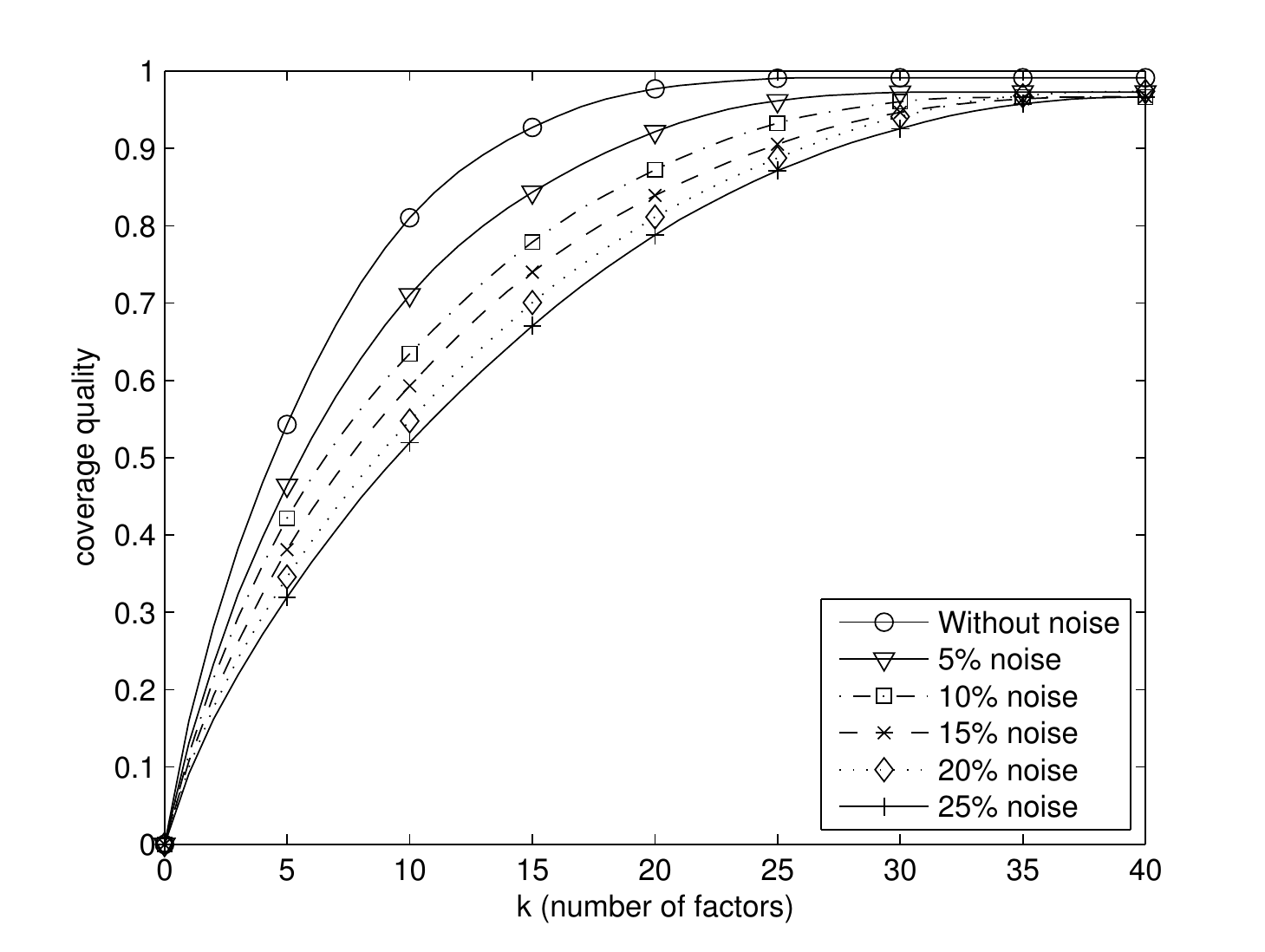}}
\subfloat[\scriptsize Aditive noise---\PaNDa]
{\label{fig:}\includegraphics[scale=0.31]{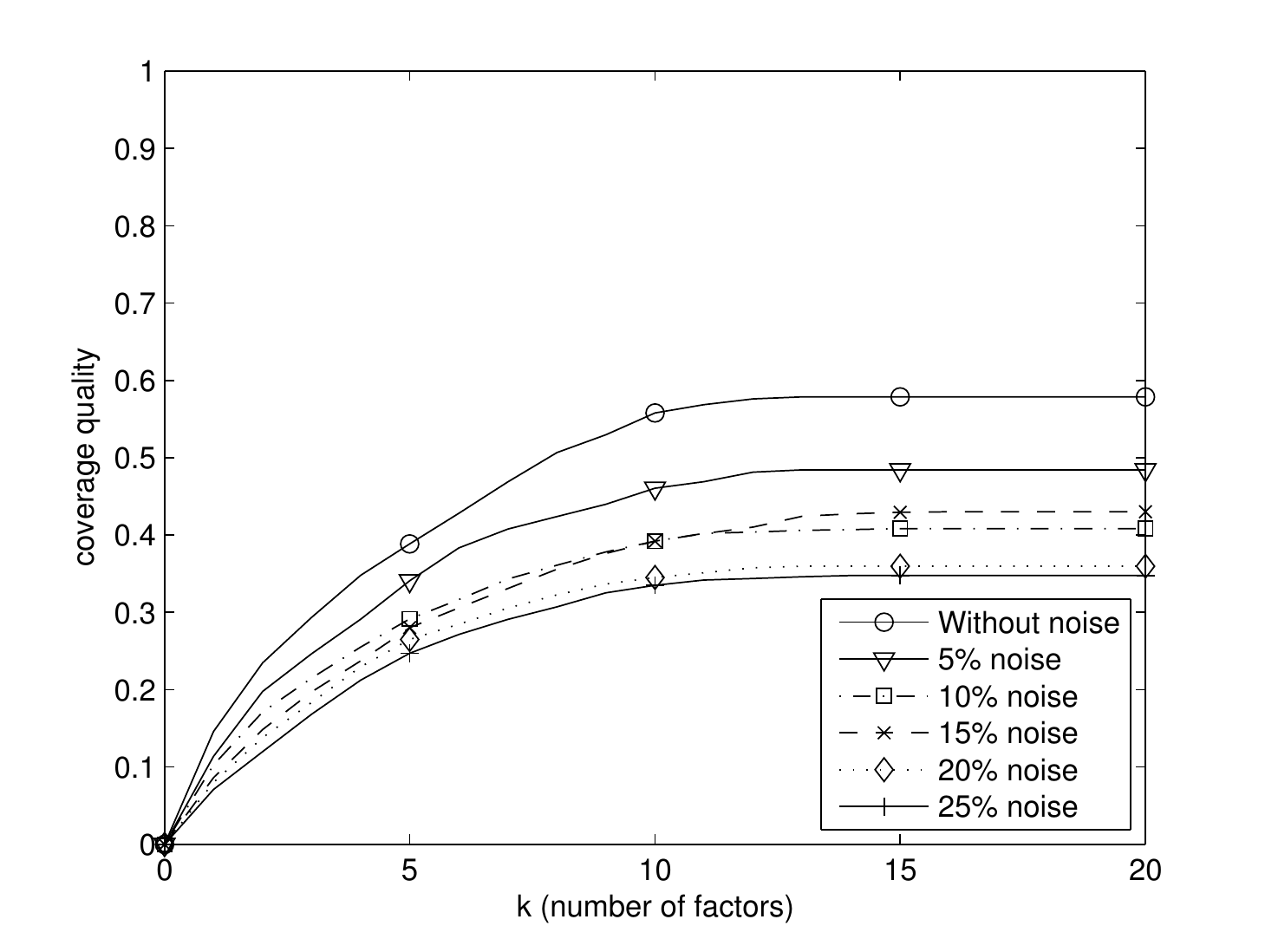}}

\subfloat[\scriptsize Substractive noise---\GreEss]
{\label{fig:}\includegraphics[scale=0.31]{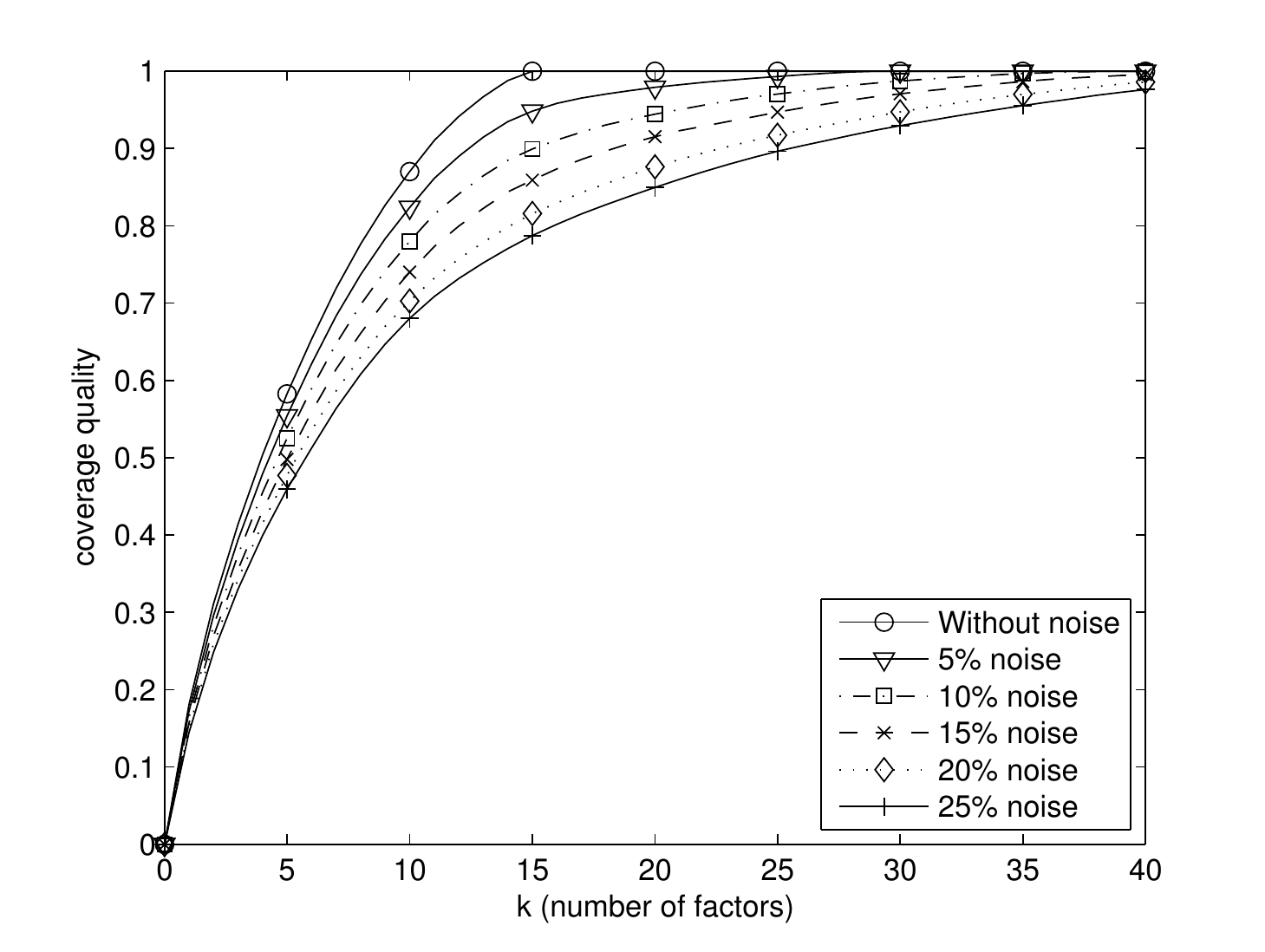}}
\subfloat[\scriptsize Substractive noise---\Asso]
{\label{fig:}\includegraphics[scale=0.31]{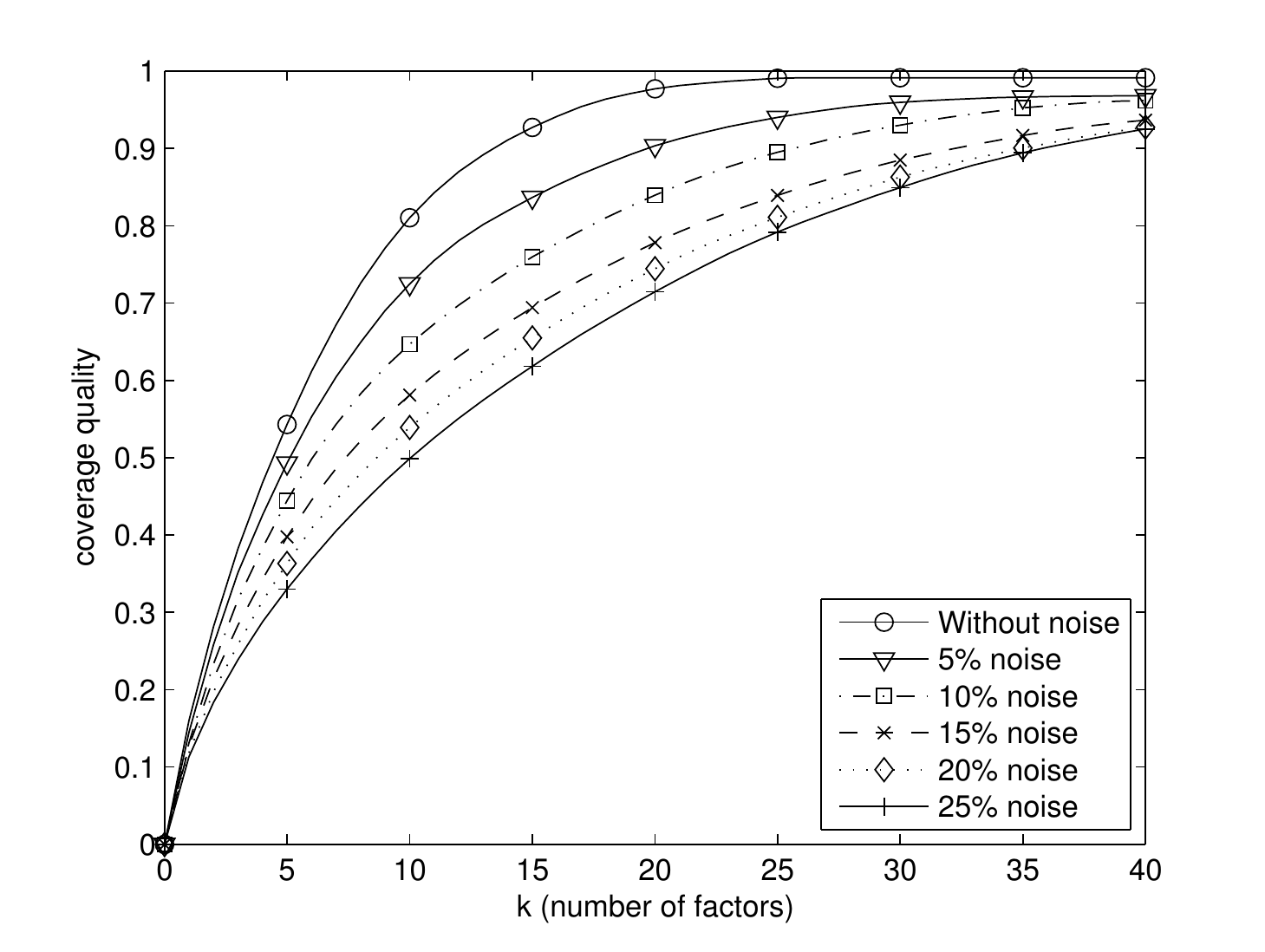}}
\subfloat[\scriptsize Substractive noise---\PaNDa]
{\label{fig:}\includegraphics[scale=0.31]{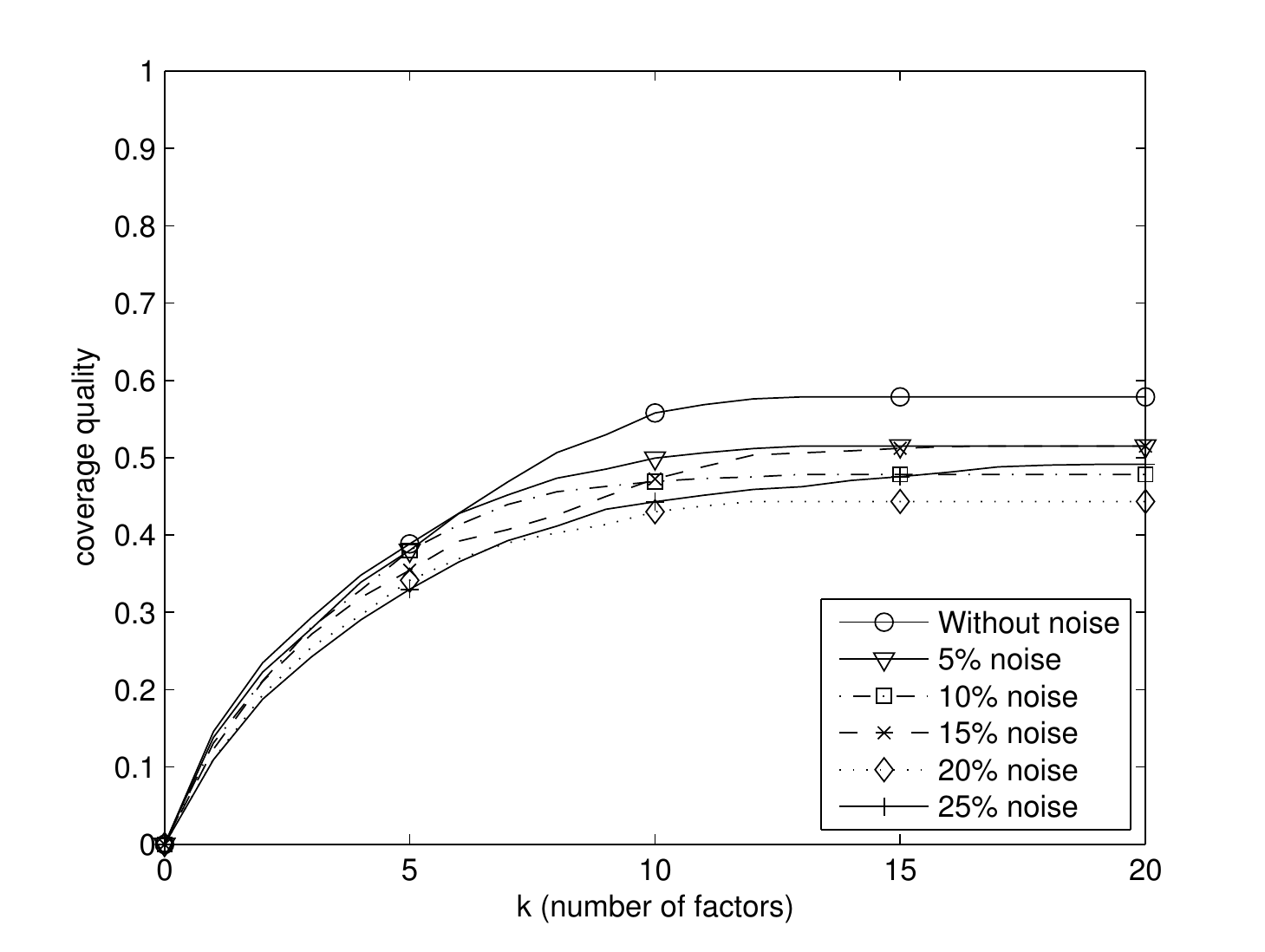}}

\subfloat[\scriptsize General noise---\GreEss]
{\label{fig:}\includegraphics[scale=0.31]{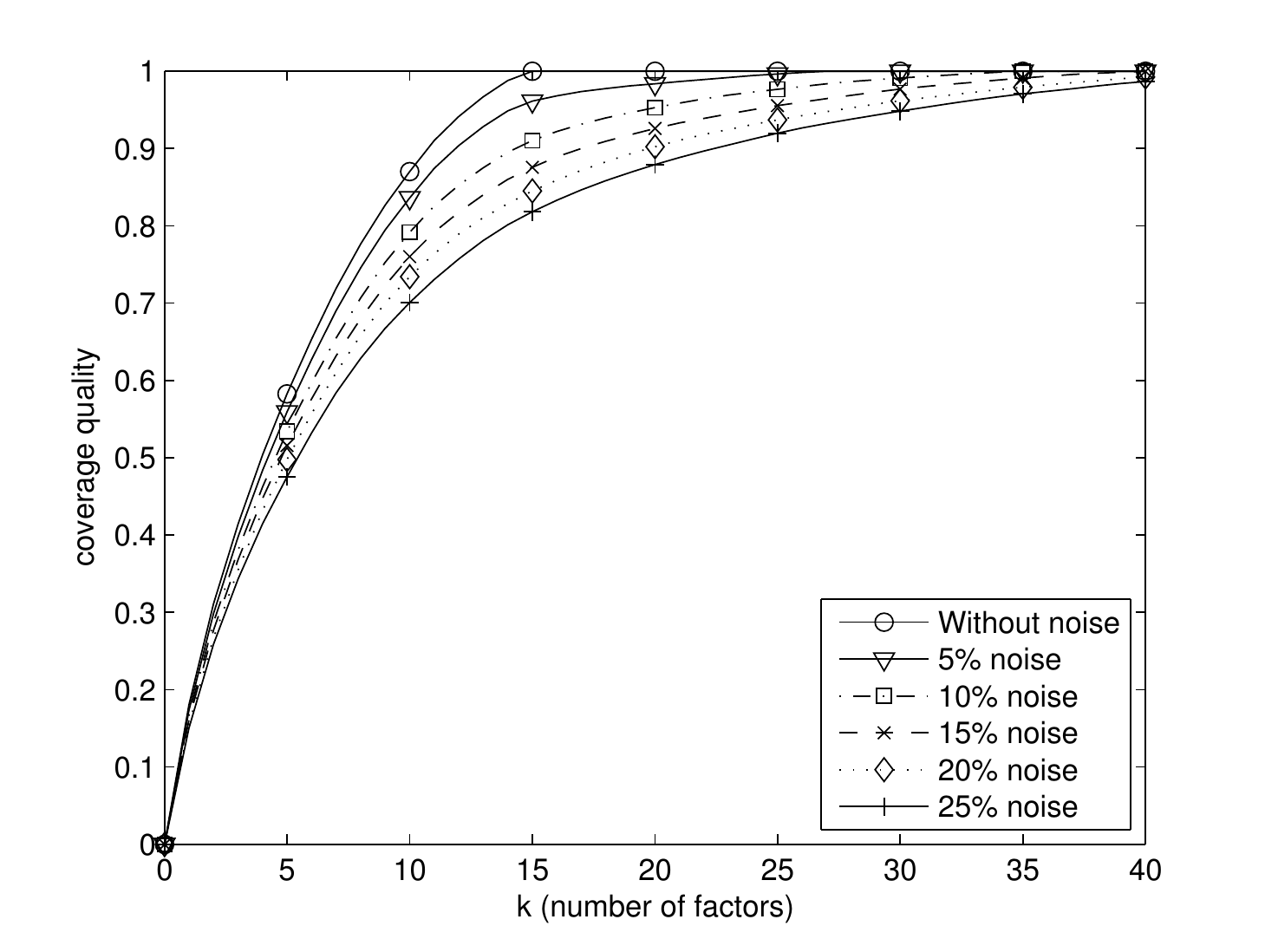}}
\subfloat[\scriptsize General noise---\Asso]
{\label{fig:}\includegraphics[scale=0.31]{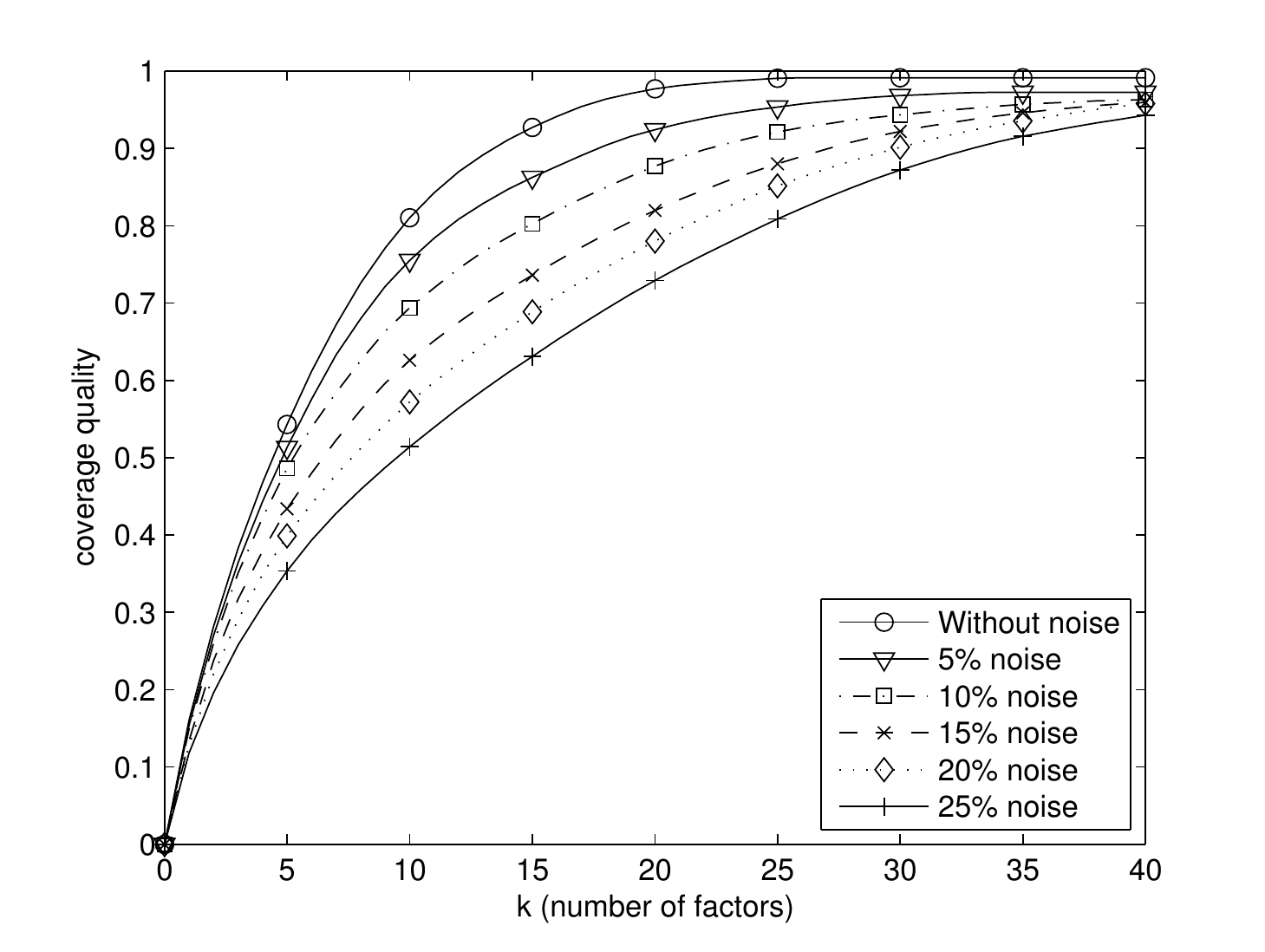}}
\subfloat[\scriptsize General noise---\PaNDa]
{\label{fig:}\includegraphics[scale=0.31]{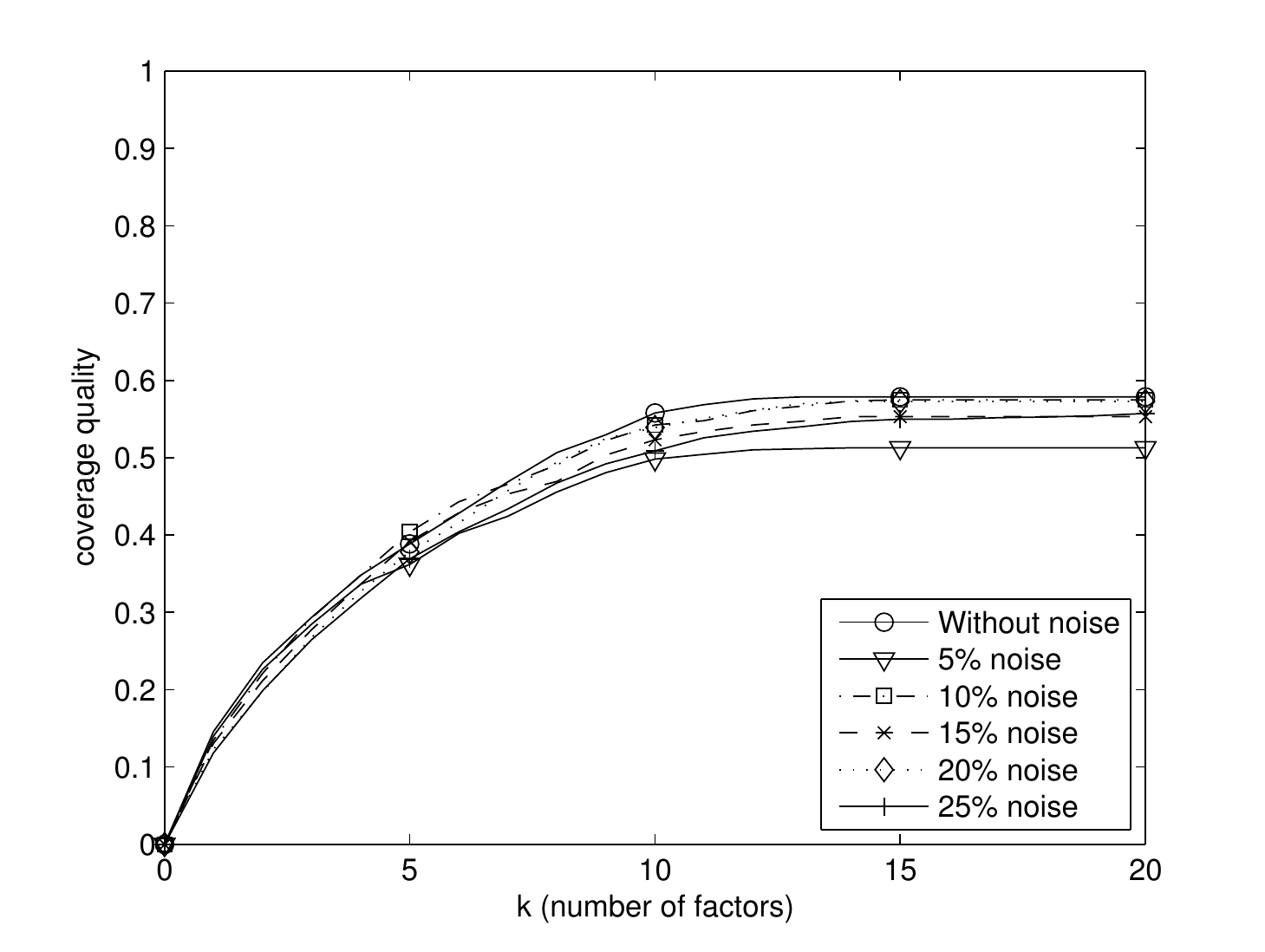}}

\caption{Coverage quality of the first $k$ factors (synthetic data with noise).}
\label{fig:cr}
\end{figure}

For each of the algorithms we provide results for additive noise, subtractive noise, 
and general noise, and for every type of noise we include six levels of noise,
from 0\% to 25\%.
Adding additive noise of $p\%$ to a Boolean matrix $I$ means that we flip at random 
with probability $p$ the entries in $I$ containing $0$ (change $0$ to $1$).
For subtractive noise, we flip the entries containing $1$ and for general noise we flip
at random all the entries.
The curves represent the coverage quality by the sets of the first $k$ factors
computed by the algorithms
as in Section \ref{sec:pa}. That is, the values for each curve are obtained as averages over 1000
particular datasets with the respective level of noise. 
We can see in Figure~\ref{fig:cr} that all the algorithms share the property that
as the level of noise increases, the curve of the coverage quality gets shifted down,
i.e. a larger number of factors is needed to explain a given portion of data.
A possible interpretation of the observed shifts is that
the larger the shift, the more sensitive the particular algorithm is for the particular type of noise.
In the context of the current view, according to which using only factors that are not allowed
to cover the entries of $I$ containing $0$ leads to sensitivity to noise,
the graphs show a somewhat surprising fact. Namely, the sensitivity to noise in the above sense for \GreEss, 
which uses only such factors covering $1$s,
turns out not to be larger than for \Asso\ and \PaNDa.
On the other hand, we believe that the current view of noise and sensitivity to noise in BMF
is limited and that the problem of noise needs a solid foundation.
For example, a natural question is whether and to what extent it is the case that  if a particular
algorithm discovers good factors in a given data, then when noise is added, the %M: preklep algorithm
algorithm still discovers
these or similar factors.
This question needs to be considered with care because adding significant amount of noise,
as in case of some experiments in the literature on BMF, may change the data to the extent that
it is much better explained by new, previously absent  factors.

%%%%%%%%%%%%%%%%%%%%%%%%%%%%%%%%%%%%%%%%%%%%
\section{Conclusions}

We presented new results on BMF that are based on examining the closure and
order-theoretic structures related to Boolean data. 
The results let us differentiate the role of entries of the input matrix and suggest
where to focus in computing  decompositions.
We proposed a new BMF algorithm, \GreEss, based on these results and provided
results of its experimental evaluation.
It turns out that %in terms of quality of the computed decompositions, 
the algorithm performs well both in terms of coverage of the input data
by the first $k$ factors (i.e. by a small number of the most important factors) 
and in terms of the number of factors needed for an exact decomposition of 
the input matrix (i.e. factors that fully explain the input data) and that \GreEss\ outperforms the
existing algorithms. % in this view.
The presented results, both theoretical and experimental,
emphasize the role of from-below factorization algorithms in BMF, of which \GreEss\ 
is an example.
%but also other two 
%algorithms, \Tiling and \GreConD, are particular examples.

An important topic for future research is to utilize further the present results
regarding essential parts of Boolean matrices and to further investigate  the 
role of entries of Boolean matrices for BMF.
In particular, it seems promising to explore the possibility to still reduce $\ess(I)$
to $\ess(\ess(I))$, and in general, to $\ess^p(I)$. 
Furthermore, the notion of $\varepsilon$-essential part shall be investigated for $\varepsilon>0$.
Another topic is to utilize as heuristics
other information that may be obtained from the intervals $\mathcal{I}_{ij}$, 
in particular the number  $|\mathcal{I}_{ij}|$ of concepts %in $\mathcal{B}(I)$ 
covering $\tu{i,j}$. This number is difficult to compute but our preliminary results indicate
that it may be approximated quickly. 
Note that the case $|\mathcal{I}_{ij}|=1$ corresponds
to so-called mandatory factors considered in \citep{BeVy:Dof}, i.e. factors that need
to be present in every exact decomposition of $I$.
%Regarding the strategies of BMF algorithms in general, our results offer the following
%perspective: MOZNA POZDEJI
An important topic is to extend the theoretical framework
%, which accounts for from-below factorizations, 
to general factorizations involving rectangles containing possibly 
$0$s, which are sometimes called fault-tolerant concepts or noisy tiles, see e.g. \citep{BePeRoBo:Cbmftpbd}.
An interesting goal is to extend the present results beyond Boolean data, namely to
 ordinal and semiring-valued data, see e.g. \citep{Bel:Odmerl} for general results on
closure structures and decompositions of such, more general data.
%\citep{BeGlVy:Oftwbdutc,Mie:Btf} consider BMF for triadic data.
Last but not least, let us mention that
%the analysis of 
three- and  multi-way data received a considerable attention recently.
\citep{BeGlVy:Oftwbdutc,Mie:Btf} present approaches to factorization of three-way Boolean
data. An extension of the present results to multi-way data seems another important research topic.

%In recent years, 
%the analysis of three- and generally multi-way data received a considerable attention.
%\citep{BeGlVy:Oftwbdutc,Mie:Btf} present approaches to factorization of three-way Boolean
%data. An  extension of the presented results to this case may bring similar improvements
%as for two-way Boolean data.

%Last but not least, the presented results on geometry of Boolean data may help

%-Frolov a NN pristup - mozna
%-Triadic

%\clearpage



%% References


%\bibliographystyle{elsarticle-num}



%M: opraveny nove reference aby sedel format

%% konec bibliografie

%TADY END
\end{document}